\documentclass[titlepage,12pt]{article} 
\usepackage{hyperref}
\usepackage{graphicx}
\usepackage[usenames,dvipsnames]{xcolor}
\usepackage{multirow}
\usepackage{amssymb,amsthm,amsmath} 
\usepackage{mathrsfs}
\usepackage[a4paper]{geometry}
\usepackage{datetime2}
\usepackage[utf8]{inputenc}
\usepackage[italian,english]{babel}

\date{}

\newcommand{\ep}{\varepsilon}
\newcommand{\re}{\mathbb{R}}

\newcommand{\ut}{\widehat{u}}
\newcommand{\ul}{u_{\lambda}}

\newcommand{\DD}{S}
\newcommand{\MM}{G}
\newcommand{\E}{\mathcal{E}}
\newcommand{\PS}{\mathcal{PS}}
\newcommand{\mmu}{\gamma}
\newcommand{\ekow}{E_{\mathrm{kow}}}
\newcommand{\etar}{E_{\mathrm{Tar}}}

\newtheorem{thm}{Theorem}[section]

\newtheorem{rmk}[thm]{Remark}
\newtheorem{prop}[thm]{Proposition}
\newtheorem{defn}[thm]{Definition}
\newtheorem{cor}[thm]{Corollary}
\newtheorem{ex}[thm]{Example}
\newtheorem{lemma}[thm]{Lemma}

\title{Generalized energy conservation for linear wave equations with time-dependent propagation speed}

\author{Marina Ghisi\vspace{1ex}\\ 
{\normalsize Università degli Studi di Pisa} \\
{\normalsize Dipartimento di Matematica}\\ 
{\normalsize PISA (Italy)}\\
{\normalsize e-mail: \texttt{marina.ghisi@unipi.it}}
\and
Massimo Gobbino\vspace{1ex}\\ 
{\normalsize Università degli Studi di Pisa} \\
{\normalsize Dipartimento di Matematica}\\ 
{\normalsize PISA (Italy)}\\  
{\normalsize e-mail: \texttt{massimo.gobbino@unipi.it}}
}

\begin{document}
\maketitle

\begin{abstract}

We consider a wave equation with a time-dependent propagation speed, whose potential oscillations are controlled through bounds on its first and second derivatives and by limiting the integral of the difference with a fixed constant. We investigate when the wave equation exhibits generalized energy conservation (GEC), meaning that the energy of all solutions remains bounded for all times by a multiple of the initial energy. When GEC is not satisfied, we provide upper bounds for the growth of the energy.

These upper bounds are derived by analyzing the growth of the Fourier components of the solution. Depending on the frequency and the time interval, different energy inequalities are employed to fully exploit our assumptions on the propagation speed.

Finally, we present counterexamples that demonstrate the optimality of our upper bound estimates.
\vspace{6ex}

\noindent{\bf Mathematics Subject Classification 2020 (MSC2020):} 
 35L20, 35L90, 35B40.

%35L20 Initial-boundary value problems for second-order hyperbolic equations
%35L90 Abstract hyperbolic equations
%35B40 - Asymptotic behavior of solutions to PDEs
%35L10 Second-order hyperbolic equations
%35L70 Second-order nonlinear hyperbolic equations
%35L72 Second-order quasilinear hyperbolic equations
%35B65 Smoothness and regularity of solutions
%35D30 Weak solutions to PDEs
		
\vspace{6ex}


\end{abstract}

%%%%%%%%%%%%%%%%%%%%%
%                   %
%   Inizio lavoro   %
%                   %
%%%%%%%%%%%%%%%%%%%%%
 
\section{Introduction}

Let us consider the wave equation
\begin{equation}
    u_{tt}(t,x)-c(t)\Delta u(t,x)=0,
    \label{eqn:wave}
\end{equation}
where the time variable $t$ belongs to the half-line $[t_0,+\infty)$, and the space variable $x$ belongs either to the whole space $\re^d$ or, more generally, to an open set $\Omega \subseteq \re^d$ (for some positive integer $d$). The coefficient $c: [t_0,+\infty) \to (0,+\infty)$ is a given function depending only on time, representing the square of the propagation speed.

We consider the usual class of weak solutions, namely functions 
\begin{equation}
    u\in C^0\left([t_0,+\infty),H^1(\re^d)\right)\cap C^1\left([t_0,+\infty),L^2(\re^d)\right)
\nonumber
\end{equation}
that satisfy (\ref{eqn:wave}) in the standard distributional sense. In the case of an open set $\Omega \subseteq \re^d$, the definition is appropriately modified to account for boundary conditions.

\paragraph{\textmd{\textit{Energy equality and generalized energy conservation}}}

In the special case where $c(t)\equiv c_0$ is a positive constant, it is well-known that the energy of the solution, defined in this case as
\begin{equation}
    E_u(t):=\|u_t(t,x)\|_{L^2(\re^d)}^2+
    c_0\cdot\|\nabla u(t,x)\|_{L^2(\re^d)}^2,
\nonumber
\end{equation}
is constant with respect to time, namely
\begin{equation}
    E_u(t)=E_u(t_0)
    \qquad
    \forall t\geq t_0.
    \label{EC}
\end{equation}

If $c$ satisfies the strict hyperbolicity condition $c(t)\geq\Lambda_1>0$ for every $t\geq t_0$, and in addition $c\in C^1([t_0,+\infty))$, then the energy of the solutions, defined in this case as
\begin{equation}
    E_u(t):=\|u_t(t,x)\|_{L^2(\re^d)}^2+
    c(t)\cdot\|\nabla u(t,x)\|_{L^2(\re^d)}^2,
    \label{defn:E-c(t)}
\end{equation}
satisfies the differential inequality
\begin{equation}
    |E_u'(t)|\leq\frac{|c'(t)|}{\Lambda_1}E_u(t)
    \qquad
    \forall t\geq t_0,
    \label{est:E'-hyp}
\end{equation}
from which one can deduce that
\begin{equation}
    E_u(t_0)\exp\left(-\frac{1}{\Lambda_1}\int_{t_0}^t|c'(\tau)|\,d\tau\right)\leq 
    E_u(t)\leq
    E_u(t_0)\exp\left(\frac{1}{\Lambda_1}\int_{t_0}^t|c'(\tau)|\,d\tau\right)
\nonumber
\end{equation}
for every $t\geq t_0$. In particular, if $c'\in L^1((t_0,+\infty))$, then all solution to (\ref{eqn:wave}) satisfy
\begin{equation}
    \frac{1}{C_0}E_u(t_0)\leq
    E_u(t)\leq
    C_0 E_u(t_0)
    \qquad
    \forall t\geq t_0,
    \label{GEC}
\end{equation}
with
\begin{equation}
    C_0:=\frac{1}{\Lambda_1}\int_{t_0}^{+\infty}|c'(\tau)|\,d\tau.
\nonumber
\end{equation}

In the literature, condition (\ref{GEC}) is referred to as \emph{generalized energy conservation} (GEC), as it serves as a substitute for the standard energy conservation law (\ref{EC}). A similar result holds if we replace the differentiability of $c$ and the integrability of $c'$ with the weaker requirement that the total variation of $c$ over $[t_0,+\infty)$ is finite.

When GEC fails, the focus shifts to estimating the potential growth from above and the potential decay from below of $E_u(t)$ as $t \to +\infty$.

Incidentally, we observe that one can define more general energies of the form
\begin{equation}
    E_u(t):=a(t)\cdot\|u_t(t,x)\|_{L^2(\re^d)}^2+
    b(t)\cdot\|\nabla u(t,x)\|_{L^2(\re^d)}^2.
    \label{defn:E-ab}
\end{equation}

If $a(t)$, $b(t)$ and $c(t)$ are bounded between two positive constants, then this energy is equivalent to the energy defined by (\ref{defn:E-c(t)}). Consequently, GEC holds with respect to the energy (\ref{defn:E-c(t)}) if and only if it holds with respect to the energy (\ref{defn:E-ab}).

\paragraph{\textmd{\textit{Previous literature}}}

When the total variation of $c(t)$ is unbounded, there are examples where GEC is violated, but also cases where GEC remains valid. The general principle is that GEC tends to fail when $c(t)$ oscillates ``too much'', as such oscillations can induce a resonance effect. Therefore, any set of conditions that sufficiently restricts the oscillations of $c(t)$ serves as a sufficient condition for GEC.

Over the last 20 years, various sets of conditions have been introduced in the literature to control these oscillations. The assumptions considered thus far generally fall into three categories.
\begin{itemize}
    \item Uniform hyperbolicity assumptions, namely asking that $c(t)$ lies in between two positive constants.

    \item Bounds on the first $m$ derivatives of $c$ (and sometimes also on the derivatives of any order).

    \item Stabilization assumptions, namely inequalities of the form
    \begin{equation}
        \int_{t_0}^t|c(\tau)-c_\infty|\,d\tau\leq S(t)
        \qquad\text{or}\qquad
        \int_{t}^{+\infty}|c(\tau)-c_\infty|\,d\tau\leq S(t)
\nonumber
    \end{equation}
    for a suitable constant $c_\infty$ and a suitable function $S(t)$.
    
\end{itemize}

A first result in this direction was obtained by M.~Reissig and J.~Smith~\cite{2005-Hokkaido-ReiSmi}. They considered a coefficient $c\in C^2([t_0,+\infty))$ for some $t_0>0$, and proved that GEC holds true provided that $c$ satisfies the uniform hyperbolicity assumption, and bounds of the form
\begin{equation}
    |c'(t)|\leq\frac{C_1}{t}
    \qquad\text{and}\qquad
    |c''(t)|\leq\frac{C_2}{t^2}
    \label{hp:yamazaki}
\end{equation}
for every $t\geq t_0$. On the opposite side, they showed that this result is optimal in the sense that there exists a coefficient $c\in C^\infty([t_0,+\infty))$ which, in addition to satisfying uniform hyperbolicity, also satisfies bounds on derivatives of any order of the form
\begin{equation}
    |c^{(k)}(t)|\leq\frac{C_{k,\beta}}{t^{\beta k}}
    \qquad
    \forall t\geq t_0,
    \quad
    \forall k\geq 1,
    \quad
    \forall\beta<1,
\nonumber
\end{equation}
and for which GEC fails.

This example motivated some authors to introduce the stabilization conditions, in the hope that they might compensate weaker assumptions on the derivatives. F.~Hirosawa~\cite{2007-MathAnn-Hirosawa} assumed that $c$ satisfies the uniform hyperbolicity, along with a bound for the derivatives of the form
\begin{equation}
    |c^{(k)}(t)|\leq\frac{C_k}{t^{\beta k}}
    \qquad
    \forall t\geq t_0,
    \quad
    \forall k\in\{1,\ldots,m\}
    \label{hp:der-hir}
\end{equation}
for some integer $m\geq 2$ and real number $\beta$, and a stabilization condition of the form
\begin{equation}
    \int_{t_0}^t|c(\tau)-c_\infty|\,d\tau\leq C_0\, t^\alpha
    \qquad
    \forall t\geq t_0
    \label{hp:stab-hir}
\end{equation}
for some real numbers $c_\infty>0$ and $\alpha\in[0,1)$ (note that the boundedness of $c(t)$ implies that the same condition is automatically true for every $\alpha\geq 1$ and every $c_\infty\in\re$).

The result, as reported in some subsequent papers (see~\cite{2015-JMAA-EbeFitHir,2021-JMAA-Hirosawa}), is that GEC holds if
\begin{equation}
    \beta\geq\beta_m:=\alpha+\frac{1-\alpha}{m},
    \label{defn:beta-m}
\end{equation}
and GEC fails if $\beta<\alpha$, while the case $\alpha\leq \beta<\beta_m$ remained open. In the same years, F.~Colombini~\cite{2006-JDE-Colombini} proved that GEC fails when $\alpha\in(0,1)$ and $\beta\leq(1-\alpha)/m$.

The validity of GEC in the range $\beta\geq\beta_m$ is a consequence of a more general estimate stated in~\cite{2007-MathAnn-Hirosawa}, namely that all solutions to (\ref{eqn:wave}) satisfy
\begin{equation}
    E_u(t)\leq E_u(t_0)\cdot\Gamma_1\exp\left(\Gamma_2\, t^{\beta_m-\beta}\right)
    \qquad
    \forall t\geq t_0
    \label{est:hirosawa}
\end{equation}
and the corresponding estimate from below, for suitable constants $\Gamma_1$ and $\Gamma_2$ that depend only on the bounds that appear in the assumptions on $c(t)$. We do not agree that (\ref{est:hirosawa}) and GEC hold true in the limit case where $\alpha=0$ and $\beta=1/m$, as well as we do not agree that GEC holds true without stabilization properties when $\beta=1/m$ and (\ref{eqn:wave}) is considered in a bounded open set (see~\cite[Theorem~1.6]{2021-JMAA-Hirosawa}). We will come back on these issues in the sequel (see Corollary~\ref{cor:alpha=0} and Example~\ref{ex:no-way}). 

Later on, M.R.~Ebert, L.~Fitriana and F.~Hirosawa in~\cite{2015-JMAA-EbeFitHir} assumed again the uniform hyperbolicity, and a bound on derivatives of the form (\ref{hp:der-hir}), but they replaced (\ref{hp:stab-hir}) with the condition
\begin{equation}
    \int_{t}^{+\infty}|c(\tau)-c_\infty|\,d\tau\leq C_0\,t^\alpha
    \qquad
    \forall t\geq t_0
    \label{hp:stab-ebert}
\end{equation}
for some real numbers $c_\infty>0$ and $\alpha<0$. Again they proved that GEC holds if $\beta\geq\beta_m$ (with $\beta_m$ given by (\ref{defn:beta-m})), and GEC fails if $\beta<\alpha$, while the case $\alpha\leq \beta<\beta_m$ remained open also in this regime.

All these results have been extended by considering more general decay/growth rates in (\ref{hp:der-hir}), (\ref{hp:stab-hir}), (\ref{hp:stab-ebert}), and by replacing the constant $c_\infty$ in the stabilization assumptions by a function $c_\infty(t)$ for which GEC is known to be true, for example a nonincreasing positive function. We refer to~\cite{2009-JMAA-HirWir,2012-Ferrara-BoiMan,2015-JMAA-EbeFitHir} for the details. Finally, it should be mentioned that in~\cite{2021-JMAA-Hirosawa} analogous results were proved for the semi-discrete version of (\ref{eqn:wave}) in which the Laplacian is replaced by finite differences. 

We point out that in many of the counterexamples exhibited so far the failure of GEC is in some sense a ``weak failure''. Indeed, what is actually proved in some papers is the existence of a sequence of times $t_k\to +\infty$, a sequence of coefficients $\{c_k\}$, and a sequence $\{u_k\}$ of corresponding solutions to (\ref{eqn:wave}) such that
\begin{equation}
    E_{u_k}(t_0)\leq 1
    \qquad\text{and}\qquad
    E_{u_k}(t_k)\to +\infty.
\nonumber
\end{equation}

In other words, different times require different solutions corresponding to different propagation speeds. In addition, many times there is a gap between the divergence rate in the counterexamples and the estimate from above in the positive results.
 
A ``strong failure'' would require a single solution $u$, corresponding to a single propagation speed, such that
\begin{equation}
    \limsup_{t\to +\infty}E_u(t)=+\infty
\nonumber
\end{equation}
with the same rate provided by the positive results.

\paragraph{\textmd{\textit{Our contribution}}}

The aim of this paper is to close the gap between the positive and negative results, and to provide counterexamples in the strong sense when GEC fails. To keep the paper at a reasonable length, we limit ourselves to the case $m=2$, and to stabilization properties of the form (\ref{hp:stab-ebert}), but we are confident that more general cases can be addressed using similar techniques. More precisely, here we focus on coefficients $c:[t_0,+\infty)\to(0,+\infty)$ of class $C^2$ that satisfy the following assumptions.
\begin{enumerate}
\renewcommand{\labelenumi}{(\arabic{enumi})}
    \item (Uniform hyperbolicity). There exists two real numbers $\Lambda_1$ and $\Lambda_2$ such that
    \begin{equation}
        0<\Lambda_1\leq c(t)\leq\Lambda_2
        \qquad
        \forall t\geq t_0.
        \label{hp:hyp}
    \end{equation}

    \item (Stabilization at infinity). There exist a real number $c_\infty$ and a non-increasing continuous function $\DD:[t_0,+\infty)\to\re$ such that $\DD(t)\to 0$ as $t\to +\infty$ and
    \begin{equation}
        \int_t^{+\infty}|c(\tau)-c_\infty|\,d\tau\leq\DD(t)
        \qquad
        \forall t\geq t_0.
        \label{hp:stab}
    \end{equation}

    We observe that $c_\infty$ is necessarily unique and belongs to the interval $[\Lambda_1,\Lambda_2]$.

    \item (Control on first two derivatives). There exists a monotone (either non-increasing or non-decreasing) function $\mmu:[t_0,+\infty)\to\re$ such that
    \begin{equation}
        |c'(t)|\leq\mmu(t)
        \quad\text{and}\quad
        |c''(t)|\leq\mmu(t)^2
        \qquad
        \forall t\geq t_0.
        \label{hp:der}
    \end{equation}
\end{enumerate}

On the positive side, in Theorem~\ref{thm:est-above} we prove an estimate from above (and also from below, see Remark~\ref{rmk:below}) for the energy of solutions to the wave equations (\ref{eqn:wave}) with a coefficient $c(t)$ that satisfies our assumptions. On the negative side, in Theorem~\ref{thm:optimality} we prove that the estimate of Theorem~\ref{thm:est-above} is sharp under rather general assumptions on $\mmu(t)$ and $S(t)$. 

In the special case where $\mmu(t)\sim t^{-\beta}$ for some $\beta\in\re$, and $S(t)\sim t^\alpha$ for some real number $\alpha<0$, the estimate of Theorem~\ref{thm:est-above} reduces to (\ref{est:hirosawa}) with $m=2$. As a corollary, we again obtain that GEC holds when $\beta\geq\beta_2$, where $\beta_2$ is defined in (\ref{defn:beta-m}) with $m=2$. In the same scenario, Theorem~\ref{thm:optimality} proves that GEC fails when $\beta<\beta_2$, thereby closing the gap left in~\cite{2015-JMAA-EbeFitHir}, with the open region now addressed by the counterexamples.

When $\alpha=0$, we again recover (\ref{est:hirosawa}) with $m=2$, but only if $\beta\neq 1/2$. However, when $\alpha=0$ and $\beta=1/2$, a logarithmic correction appears in the argument of the exponential, preventing us from concluding that GEC holds in this case. In fact, Theorem~\ref{thm:optimality} shows that GEC fails whenever $\alpha=0$ and $\beta\leq 1/2$, demonstrating that the logarithmic correction is crucial.

Whenever we prove that GEC fails, it is always a strong failure, namely there exist a propagation speed that satisfies the assumptions, and a corresponding solution to the wave equation (\ref{eqn:wave}) whose energy grows as much as permitted by the bounds of Theorem~\ref{thm:est-above} (up to constants).

Our results are valid also for abstract evolution equations of the form
\begin{equation}
    u''(t)+c(t)Au(t)=0,
    \label{eqn:main}
\end{equation}
where $A$ is a non-negative self-adjoint operator with dense domain $D(A)$ in a real Hilbert space $\mathcal{H}$. Here we stress that, under our assumptions, the failure of GEC is always driven by large frequencies, and therefore the result holds true also in bounded domains.

\paragraph{\textmd{\textit{Overview of the technique}}}

Let us consider the family of second order ordinary differential equations
\begin{equation}
    \ul''(t)+\lambda^2 c(t)\ul(t)=0,
    \label{eqn:ode}
\end{equation}
depending on the real parameter $\lambda>0$. Thanks to Fourier analysis, it is well-known that proving energy estimates for solution to (\ref{eqn:wave}) is equivalent to proving $\lambda$-independent estimates for solutions to (\ref{eqn:ode}). 

In the proof of Theorem~\ref{thm:est-above} we follow this path, and we obtain the required $\lambda$-independent estimates by demonstrating that the appropriate energies satisfy specific differential inequalities. This method differs from previous works in the literature, which relied on phase space analysis and diagonalization techniques.

The construction of the counterexamples adapts the classical procedure introduced in the seminal paper~\cite{DGCS}, which involves finding a coefficient $c(t)$ such that (\ref{eqn:ode}) admits a family of solutions that grow exponentially with respect to both time and $\lambda$. The process typically consists of two steps.
\begin{itemize}
    \item In the first step, a special solution is used to construct a coefficient $c(t)$ that, within a given time-interval $[a,b]$, exponentially amplifies all components of $u$ corresponding to frequencies near a specific $\lambda$. In other words, these components are of order~1 at time $t=a$, and grow to the maximum order allowed by the positive result by time $t=b$. When this occurs, we say that $c(t)$ activates the frequency $\lambda$. In our case, this first step is accomplished in Proposition~\ref{prop:main}.

    \item In the second step, an iterative process is designed to produce a coefficient that activates the frequency $\lambda_k$ over a sequence of intervals $[a_k,b_k]$. By the end, at each time $t=b_k$, the solution’s energy is as large as desired, though the components responsible for this growth vary with each $k$. This second step is carried out in section~\ref{sec:iteration}.

\end{itemize}

It is noteworthy that there is an interplay between the positive result and the counterexamples: the parameter choices in the proof of Theorem~\ref{thm:est-above}, and specifically which energy is used in which time interval, guide the parameter selection in the construction of the counterexamples.

\paragraph{\textmd{\textit{A related problem}}}

The problem of estimating the growth at infinity of the energy parallels the problem of local existence of solutions. To illustrate this parallel, let us now consider equation (\ref{eqn:wave}) with a coefficient $c(t)$ defined in some interval $(0,t_0)$. On $c(t)$ we can impose uniform hyperbolicity conditions such as (\ref{hp:hyp}), just for $t\in(0,t_0)$, and a control on the derivatives of the form (\ref{hp:der}) within the same interval. Regarding stabilization properties, these correspond to assumptions such as
\begin{equation}
    \int_0^t|c(\tau)-c_\infty|\,d\tau\leq C_0\, t^\alpha
    \qquad
    \forall t\in(0,t_0),
    \label{hp:stab-0}
\end{equation}
which now make sense if $\alpha>1$ (otherwise they are a consequence of the boundedness of $c$). We observe that (\ref{hp:stab-0}) can be interpreted as some sort of ``integral Hölder continuity in~0'', placing us close to the classical scenario where one controls both the modulus of continuity and the potential blow-up of derivatives at the origin (see for example~\cite{2003-BSM-ColDSaRei,2007-DIE-DSaKinRei,gg:CDSR-optimal,gg:OptDerLoss} and the references quoted therein).

In this context, the equivalent of GEC is well-posedness in Sobolev spaces, because both are equivalent to $\lambda$-independent estimates for solutions to (\ref{eqn:ode}). For example, the fact that GEC holds when $c(t)$ satisfies (\ref{hp:yamazaki}) at infinity parallels the classical result that (\ref{eqn:wave}) is well-posed in Sobolev spaces when $c(t)$ satisfies (\ref{hp:yamazaki}) in a right neighborhood of the origin (see \cite{1990-CPDE-Yamazaki}). Similarly, an exponential growth of the energy corresponds to well-posedness in suitable Gevrey spaces.

Therefore, it is hardly surprising that the statements and proofs related to GEC are analogous to those for well-posedness.

\paragraph{\textmd{\textit{Structure of the paper}}}

This paper is organized as follows. In section~\ref{sec:statements} we state our main results and we compare them with previous literature in the case of powers. In section~\ref{sec:positive} we prove our estimates from above. In section~\ref{sec:negative} we construct the counterexamples that show their optimality.

%%\clearpage

\setcounter{equation}{0}
\section{Statements}\label{sec:statements}

\subsection{Notation and definitions}

In this paper we consider the abstract evolution equation (\ref{eqn:main}), where $A$ is a linear operator in a real Hilbert space $\mathcal{H}$. We always assume that $A$ is unitary equivalent to a non-negative multiplication operator in some $L^{2}$ space. More precisely, we assume that there exist a measure space $(\mathcal{M},\mu)$, a measurable function $\lambda:\mathcal{M}\to[0,+\infty)$, and a linear isometry $\mathscr{F}:H\to L^{2}(\mathcal{M},\mu)$ with the property that for every $u\in H$ it turns out that
\begin{equation}
u\in D(A)
\quad\Longleftrightarrow\quad
\lambda(\xi)^{2}[\mathscr{F}u](\xi)\in L^{2}(\mathcal{M},\mu),
\nonumber
\end{equation}
and for every $u\in D(A)$ it turns out that
\begin{equation}
\left[\mathscr{F}(Au)\right](\xi)=
\lambda(\xi)^{2}[\mathscr{F}u](\xi)
\qquad
\forall\xi\in\mathcal{M}. 
\nonumber
\end{equation}

Roughly speaking, $\mathscr{F}$ is a sort of generalized Fourier transform that allows us to identify each element $u\in\mathcal{H}$ with a function $\ut\in L^{2}(\mathcal{M},\mu)$, and under this identification the operator $A$ becomes the multiplication operator by $\lambda(\xi)^{2}$ in $L^{2}(\mathcal{M},\mu)$. Under this identification, the spectrum of $A$ can be defined as follows.

\begin{defn}[Spectrum of a multiplication operator]\label{defn:spectrum}
\begin{em}
    Let $\mathcal{H}$ be a Hilbert space, and let $A$ be a multiplication operator on $\mathcal{H}$, defined as above. The \emph{spectrum} of $A$ is the set of all real numbers $\lambda^2$ such that
    \begin{equation}
        \forall r>0
        \qquad
        \mu\left(\{\xi\in\mathcal{M}:\lambda(\xi)^2\in[\lambda^2-r,\lambda^2+r]\}\right)>0.
\nonumber
    \end{equation}
    
\end{em}
\end{defn}

Let us introduce the classes of coefficients $c(t)$ that we consider in (\ref{eqn:main}).

\begin{defn}[Admissible propagation speeds]\label{defn:PS}
\begin{em}
    Let $t_0$ and $0<\Lambda_1\leq\Lambda_2$ be three real numbers, let $\DD:[t_0,+\infty)\to\re$ be a continuous non-increasing function such that $\DD(t)\to 0$ as $t\to +\infty$, and let $\mmu:[t_0,+\infty)\to\re$ be a continuous monotone function. 
    
    We call $\PS(t_0,\Lambda_1,\Lambda_2,\DD,\mmu)$ the set of all functions $c:[t_0,+\infty)\to\re$ of class $C^2$ that satisfy (\ref{hp:hyp}), (\ref{hp:stab}) and (\ref{hp:der}). 
    
    We call $S$ the \emph{stabilization rate} of $c$, and $\gamma$ the \emph{growth/decay} (depending on the verse of monotonicity) \emph{rate} of its derivatives. For future use, we introduce the non-decreasing and non-negative continuous functions
\begin{equation}
    \MM(t):=\int_{t_0}^t\mmu(\tau)^2\,d\tau
    \qquad
    \forall t\geq t_0,
    \label{defn:G}
\end{equation}
and
\begin{equation}
    M(t):=\max\{\MM(\tau)\DD(\tau):\tau\in[t_0,t]\}
    \qquad
    \forall t\geq t_0.
    \label{defn:Max}
\end{equation}

\end{em}
\end{defn}

For all solutions $u$ to (\ref{eqn:main}), we fix once for all, among all equivalent possibilities, the energy
\begin{equation}
    \E_u(t):=\frac{\|u'(t)\|_H^2}{c_\infty^{1/2}}+
    c_\infty^{1/2}\cdot\|A^{1/2}u(t)\|_H^2
    \qquad
    \forall t\geq t_0,
    \label{defn:energy}
\end{equation}
where $c_\infty$ is the constant that appears in the stabilization assumption (\ref{hp:stab}).

We are now ready to define formally what it means for GEC to hold in a class of coefficients.

\begin{defn}[Generalized energy conservation]\label{defn:GEC}
\begin{em}

    Let $t_0$, $\Lambda_1$, $\Lambda_2$, $\DD$ and $\mmu$ be as in Definition~\ref{defn:PS}.  
    
    We say that the \emph{generalized energy conservation} holds true in $\PS(t_0,\Lambda_1,\Lambda_2,\DD,\mmu)$ if there exists a positive real number $H_0$ with the following property. For every Hilbert space $\mathcal{H}$, every multiplication operator $A$ on $\mathcal{H}$, every coefficient $c\in\PS(t_0,\Lambda_1,\Lambda_2,\DD,\mmu)$, and every solution $u$ of (\ref{eqn:main}), it turns out that
    \begin{equation}
        \frac{1}{H_0}\E_u(t_0)\leq \E_u(t)\leq H_0 \E_u(t_0)
        \qquad
        \forall t\geq t_0.
\nonumber
    \end{equation}
\end{em}
\end{defn}

\subsection{Positive result -- Estimate of energy growth}

The first main result of this paper is the following energy estimate for solutions to (\ref{eqn:main}) (cfr.~\cite[Theorem~2.4]{2015-JMAA-EbeFitHir}).

\begin{thm}[Estimates from above]\label{thm:est-above}

Let $t_0$, $\Lambda_1$, $\Lambda_2$, $\DD$ and $\mmu$ be as in Definition~\ref{defn:PS}. Let $\MM$ and $M$ be the functions defined in (\ref{defn:G}) and (\ref{defn:Max}), respectively.

Then for every Hilbert space $\mathcal{H}$, every multiplication operator $A$ on $\mathcal{H}$, every coefficient $c\in\PS(t_0,\Lambda_1,\Lambda_2,\DD,\mmu)$, and every solution $u$ to (\ref{eqn:main}), the energy $\E_u(t)$ defined by (\ref{defn:energy}) satisfies the following estimates.

\begin{enumerate}
\renewcommand{\labelenumi}{(\arabic{enumi})}
    \item If $\mmu$ is non-decreasing, then 
    \begin{equation}
        \E_u(t)\leq\E_u(t_0)\cdot H_1\exp\left(H_2 M(t)^{1/2}\right)
        \qquad
        \forall t\geq t_0,
        \label{th:mu-up}
    \end{equation}
    where
    \begin{equation}
        H_1:=\frac{3\Lambda_2}{\Lambda_1},
        \qquad\qquad
        H_2:=\max\left\{\frac{2(\Lambda_2-\Lambda_1)^{1/2}}{\Lambda_1^2},
        \frac{(2\Lambda_1+3)^{1/2}}{\Lambda_1^{3/2}}\right\}.
        \label{defn:C12}
    \end{equation}

    \item If $\mmu$ is non-increasing, then 
    \begin{equation}
        \E_u(t)\leq\E_u(t_0)\cdot H_3\exp\left(H_4 M(t)^{1/2}\right)
        \qquad
        \forall t\geq t_0,
        \label{th:mu-down}
    \end{equation}
    where
    \begin{equation}
        H_3:=\max\left\{\frac{3\Lambda_2}{\Lambda_1},\exp\left(\frac{\mmu(t_0)\DD(t_0)}{2\Lambda_1^2}\right)\right\},
        \qquad\qquad
        H_4:=\frac{(2\Lambda_1+3)^{1/2}}{\Lambda_1^{3/2}}.
        \label{defn:C34}
    \end{equation}

\end{enumerate}

\end{thm}

\begin{rmk}[Estimates from below]\label{rmk:below}
\begin{em}

The corresponding estimates from below hold true as well. For example, when $\mu$ is non-increasing, it turns out that
\begin{equation}
    \E_u(t)\geq\E_u(t_0)\cdot (H_1)^{-1}\exp\left(-H_2 M(t)^{1/2}\right)
    \qquad
    \forall t\geq t_0,
\nonumber
\end{equation}
and analogously in the case of (\ref{th:mu-down}).

\end{em}
\end{rmk}

We conclude by observing that GEC holds whenever $M(t)$ is bounded.

\begin{cor}[Sufficient condition for the generalized energy conservation]

    Under the same assumptions of Theorem~\ref{thm:est-above}, the generalized energy conservation holds true in $\PS(t_0,\Lambda_1,\Lambda_2,\DD,\mmu)$ if
    \begin{equation}
        \sup\left\{\MM(\tau)\DD(\tau):\tau\geq t_0\right\}<+\infty.
\nonumber
    \end{equation}
\end{cor}

\subsection{Negative result -- Counterexamples}

The second main result of this paper addresses the extent to which the estimates in Theorem~\ref{thm:est-above} are optimal. To start, it’s important to note that these estimates cannot be universally optimal, as demonstrated by the following example.

\begin{ex}[Fast growth of derivatives]\label{ex:non-opt}
\begin{em}
    Let us consider the case where 
    $$t_0:=0, 
    \qquad\qquad
    \mmu(t):=2\sqrt{t}\exp(t^2),
    \qquad\qquad
    S(t):=\frac{1}{\sqrt{t+1}}.$$

    An easy calculation shows that
    \begin{equation}
        G(t)=\int_0^t 4\tau\exp(2\tau^2)\,d\tau=
        \exp(2t^2)-1
        \qquad
        \forall t\geq 0,
    \nonumber
    \end{equation}
    and hence
    \begin{equation}
        M(t)=
        \max\{G(\tau)S(\tau):\tau\in[0,t]\}\sim
        \frac{\exp(2t^2)}{\sqrt{t}}
        \qquad
        \text{as }t\to +\infty.
    \nonumber
    \end{equation}
    
    Consequently, in this case Theorem~\ref{thm:est-above} gives an upper bound of the form
    \begin{equation}
        \E_u(t)\lesssim\E_u(0)\cdot H_1
        \exp\left(H_2 t^{-1/4}\exp(t^2)\right)
        \qquad
        \text{as }t\to +\infty.
        \label{est:ex-tar}
    \end{equation}

    On the other hand, if we consider the classical hyperbolic energy
    \begin{equation}
        \widehat{\E}_u(t):=\|u'(t)\|_H^2+
        c(t)\cdot\|A^{1/2}u(t)\|_H^2,
    \nonumber
    \end{equation}
    then one can show that it satisfies an estimate analogous to (\ref{est:E'-hyp}), from which one can deduce that
    \begin{equation}
        \widehat{\E}_u(t)\leq\widehat{\E}_u(0)\cdot
        \exp\left(\frac{1}{\Lambda_1}\int_0^t 2\sqrt{\tau}\exp(\tau^2)\,d\tau\right).
        \label{est:ex-hyp}
    \end{equation}

    Since the two energies are equivalent, and
    \begin{equation}
        \int_0^t 2\sqrt{\tau}\exp(\tau^2)\,d\tau\sim
        \frac{\exp(t^2)}{\sqrt{t}}
        \qquad
        \text{as }t\to +\infty,
    \nonumber
    \end{equation}
    it is clear that (\ref{est:ex-hyp}) is much better than (\ref{est:ex-tar}), at least for $t$ large enough, and consequently (\ref{est:ex-tar}) can not be optimal in this case.
\end{em}    
\end{ex}

A careful analysis of Example~\ref{ex:non-opt} above reveals that its effectiveness relies on the more than exponential growth of $\mmu(t)$ at infinity, leading to its integral growing more slowly than $\mmu(t)$. This observation suggests that, to prove the optimality of Theorem~\ref{thm:est-above}, we should focus on functions $\mmu$ that grow at most exponentially as $t\to +\infty$. This is precisely the approach taken in the next result.

\begin{thm}[Optimality of the estimates from above]\label{thm:optimality}
    Let $t_0$, $\Lambda_1$, $\Lambda_2$, $\DD$ and $\mmu$ be as in Definition~\ref{defn:PS}. Let $\MM$ and $M$ be the functions defined in (\ref{defn:G}) and (\ref{defn:Max}), respectively. Let $A$ be a multiplication operator in a real Hilbert space $\mathcal{H}$.

    Let us assume in addition that
    \begin{enumerate}
    \renewcommand{\labelenumi}{(\roman{enumi})}
        \item the spectrum of $A$ contains the half-line $(\lambda_0^2,+\infty)$ for some real number $\lambda_0\geq 0$,

        \item  $\Lambda_2>\Lambda_1$, and
        \begin{equation}
            \sup\left\{\MM(\tau)\DD(\tau):\tau\geq t_0\right\}=+\infty,
    \nonumber
        \end{equation}

        \item $\gamma\in C^3([t_0,+\infty))$, and there exists a constant $\Lambda_3$ such that
        \begin{equation}
            \max\{|\gamma'(t)|,|\gamma''(t)|,|\gamma'''(t)|\}\leq
            \Lambda_3\gamma(t)
            \qquad
            \forall t\geq t_0.
            \label{hp:gamma-C3}
        \end{equation}
    \end{enumerate}

    Then there exist a coefficient $c\in\PS(t_0,\Lambda_1,\Lambda_2,\DD,\mmu)$, a nonzero solution $u(t)$ to (\ref{eqn:main}), a positive real number $H_5$, and a sequence $\{b_k\}\subseteq[t_0,+\infty)$ with $b_k\to +\infty$ such that
    \begin{equation}
        \E_u(b_k)\geq\E_u(t_0)\cdot\exp\left(H_5 M(b_k)^{1/2}\right)
        \qquad
        \forall k\geq 1.
        \label{th:counterexample}
    \end{equation}
\end{thm}

We conclude with two brief comments on the assumptions and the conclusion of Theorem~\ref{thm:optimality}.

\begin{rmk}\label{rmk:spectrum}
\begin{em}
    Let us discuss the assumption regarding the spectrum of $A$.
    
    If the function $G(t)S(t)$ is non-decreasing, and hence $M(t)=G(t)S(t)$ for every $t\geq t_0$, then assumption (i) in Theorem~\ref{thm:optimality} in not required, and it is enough that the spectrum of $A$ is unbounded. This scenario applies whenever $S(t)$ and $\mmu(t)$ are powers of $t$, since in that case the function $G(t)S(t)$ is unbounded if and only if it is increasing (see the discussion in section~\ref{sec:powers} below).

    When $G(t)S(t)$ is not necessarily monotone, we can still replace assumption~(i) by the weaker requirement that there exist two real numbers $\lambda_0\geq 0$ and $\Gamma_0>1$ with the property that, for every $\lambda\geq\lambda_0$, the interval $(\lambda,\Gamma_0\lambda)$ intersects the spectrum of $A$.  Notably, this condition is met when $A$ is (minus) the Dirichlet Laplacian in a bounded open set.

    Both variants can be proved with a minor adjustment to the definition of $\lambda$ in the proof of Proposition~\ref{prop:main}.
    
\end{em}    
\end{rmk}

\begin{rmk}
\begin{em}
    The conclusion of Theorem~\ref{thm:optimality} can be further strengthened by showing that the set of coefficients $c(t)$ that serve as counterexamples is actually residual (in the sense of Baire category) with respect to a suitable metric in the set of all admissible propagation speeds. This technique has been applied numerous times in previous works. For further details, see~\cite{gg:residual,gg:DGCS-critical,gg:OptDerLoss}.
\end{em}    
\end{rmk}

\subsection{Some classical examples}\label{sec:powers}

Let us examine the consequences of our results in some special cases. Let us start with the case where $t_0>0$ and
\begin{equation}
    \mmu(t):=t^{-\beta}
    \qquad\text{and}\qquad
    \DD(t):=t^\alpha
    \label{defn:powers}
\end{equation}
with $\beta\in\re$ and $\alpha<0$. If $\beta\neq 1/2$, then
\begin{equation}
    G(t)=\frac{1}{1-2\beta}\left(t^{1-2\beta}-t_0^{1-2\beta}\right),
    \nonumber
\end{equation}
and hence
\begin{equation}
    G(t)\DD(t)=
    \frac{1}{1-2\beta}\left(t^{1-2\beta}-t_0^{1-2\beta}\right)t^\alpha.
    \label{eqn:Mt}
\end{equation}

If $2\beta\geq 1+\alpha$, this function is bounded for $t\geq t_0$. If $2\beta< 1+\alpha$, this function is increasing, and therefore
\begin{equation}
    M(t)=G(t)\DD(t)\sim t^{1+\alpha-2\beta}
    \qquad
    \text{as }t\to +\infty.
    \nonumber
\end{equation}

The case $\beta=1/2$ is special because 
\begin{equation}
    G(t)=\log\left(\frac{t}{t_0}\right)
    \qquad\text{and}\qquad
    G(t)\DD(t)=t^\alpha\log\left(\frac{t}{t_0}\right),
    \label{eqn:Mt-log}
\end{equation}
so that $G(t)\DD(t)$ is bounded from above for all $\alpha<0$.

Thus from Theorem~\ref{thm:est-above}, Theorem~\ref{thm:optimality} and Remark~\ref{rmk:spectrum} we deduce the following result, which can be summed up as  ``GEC holds if and only if $2\beta\geq 1+\alpha$''.

\begin{cor}[The case of powers]
    Let $t_0$ and $\Lambda_1\leq\Lambda_2$ be three positive real numbers, and let $\mmu(t)$ and $\DD(t)$ be defined by (\ref{defn:powers}) with $\beta\in\re$ and $\alpha<0$.
    
    Then the following statements hold true.
    \begin{enumerate}
    \renewcommand{\labelenumi}{(\arabic{enumi})}
        \item If $2\beta\geq 1+\alpha$, then  GEC holds true in the class $\PS(t_0,\Lambda_1,\Lambda_2,\DD,\mmu)$.

        \item If $2\beta< 1+\alpha$, then for every $c\in\PS(t_0,\Lambda_1,\Lambda_2,\DD,\mmu)$ it turns out that all solutions $u$ to (\ref{eqn:main}) satisfy
        \begin{equation}
            \E_u(t)\leq\E_u(t_0)\cdot
            H_6\exp\left(H_7 t^{(\alpha+1)/2-\beta}\right)
            \qquad
            \forall t\geq t_0
    \nonumber
        \end{equation}
        for suitable constants $H_6$ and $H_7$ that depend only on $t_0$, $\Lambda_1$, $\Lambda_2$, $\alpha$, $\beta$.

        \item If $2\beta< 1+\alpha$, and in addition $\Lambda_2>\Lambda_1$ and the spectrum of $A$ is unbounded, then there exist a coefficient $c\in\PS(t_0,\Lambda_1,\Lambda_2,\DD,\mmu)$, a solution $u$ to (\ref{eqn:main}), and a positive real number $H_8$ such that
        \begin{equation}
            \limsup_{t\to+\infty}
            \E_u(t)\exp\left(-H_8 t^{(\alpha+1)/2-\beta}\right)=+\infty.
    \nonumber
        \end{equation}
    \end{enumerate}

\end{cor}

Let us consider now the limit case $\alpha=0$. This scenario technically falls outside the framework of this paper, since we assumed from the very beginning that $S(t)\to 0$ as $t\to +\infty$. On the other hand, any function that satisfies (\ref{hp:stab}) with $S(t)\equiv C_0$ also satisfies the same estimate with some $\widehat{S}(t)$ that tends to~0 at infinity.

If $\beta\neq 1/2$, setting  $\alpha=0$ in (\ref{eqn:Mt}) we deduce that $M(t)$ is bounded if $\beta>1/2$, and $M(t)\sim t^{1-2\beta}$ if $\beta<1/2$. If $\beta=1/2$, setting $\alpha=0$ in (\ref{eqn:Mt-log}) we deduce that $M(t)=\log(t/t_0)$ is now unbounded, in contrast with the case $\alpha<0$. In both cases the corresponding growth rate provided by Theorem~\ref{thm:est-above} can not be optimal because it does not fully account for the ``true'' stabilization rate of the coefficient $c(t)$. 

The precise result is stated in Corollary~\ref{cor:alpha=0} below. 
The short version is that GEC holds if and only if $\beta>1/2$. We stress that now GEC fails also for $\beta=1/2$.

\begin{cor}[The limit case $\alpha=0$]\label{cor:alpha=0}  
    Let $t_0$ and $\Lambda_1\leq\Lambda_2$ be three positive real numbers, and let $\mmu(t):=t^{-\beta}$ for some $\beta\in\re$. Let $\MM$ and $M$ be the functions defined in (\ref{defn:G}) and (\ref{defn:Max}), respectively. 
    
    Then the following statements hold true.
    \begin{enumerate}
    \renewcommand{\labelenumi}{(\arabic{enumi})}
        \item If $\beta>1/2$, then  GEC holds true in the class $\PS(t_0,\Lambda_1,\Lambda_2,\DD,\mmu)$ for every choice of the stabilization rate $S$.

        \item If $\beta<1/2$, then for every stabilization rate $\DD$, and for every coefficient $c\in\PS(t_0,\Lambda_1,\Lambda_2,\DD,\mmu)$, it turns out that all solutions $u$ to (\ref{eqn:main}) satisfy
        \begin{equation}
            \E_u(t)\leq\E_u(t_0)\cdot
            H_{9}\exp\left(H_{10}t^{1/2-\beta}\right)
            \qquad
            \forall t\geq t_0
    \nonumber
        \end{equation}
        for suitable constants $H_{9}$ and $H_{10}$ that depend only on $t_0$, $\Lambda_1$, $\Lambda_2$, $\beta$, and $\DD(t_0)$.

        \item If $\beta=1/2$, then for every stabilization rate $\DD$, and for every coefficient $c\in\PS(t_0,\Lambda_1,\Lambda_2,\DD,\mmu)$, it turns out that all solutions $u$ to (\ref{eqn:main}) satisfy
        \begin{equation}
            \E_u(t)\leq\E_u(t_0)\cdot
            H_{11}\exp\left(H_{12}|\log t|^{1/2}\right)
            \qquad
            \forall t\geq t_0
     \nonumber
       \end{equation}
        for suitable constants $H_{11}$ and $H_{12}$ that depend only on $t_0$, $\Lambda_1$, $\Lambda_2$, and $\DD(t_0)$.

        \item Let us assume that $\beta< 1/2$, and in addition $\Lambda_2>\Lambda_1$ and the spectrum of $A$ contains the half-line $(\lambda_0^2,+\infty)$ for some real number $\lambda_0\geq 0$. Let $\DD$ be a stabilization rate such that
        \begin{equation}
            \sup\left\{\DD(\tau)\tau^{1-2\beta}:\tau\geq t_0\right\}=+\infty.
     \nonumber
       \end{equation}
        
        Then there exist a coefficient $c\in\PS(t_0,\Lambda_1,\Lambda_2,\DD,\mmu)$, a solution $u$ to (\ref{eqn:main}), and a positive real number $H_{13}$ such that
        \begin{equation}
            \limsup_{t\to+\infty}
            \E_u(t)\exp\left(-H_{13} M(t)^{1/2}\right)=+\infty.
             \label{th:no-GEC}
        \end{equation}

        \item Let us assume that $\beta=1/2$, and in addition $\Lambda_2>\Lambda_1$ and the spectrum of $A$ contains the half-line $(\lambda_0^2,+\infty)$ for some real number $\lambda_0\geq 0$. Let $\DD$ be a stabilization rate such that
        \begin{equation}
            \sup\left\{\DD(\tau)\log(\tau):\tau\geq t_0\right\}=+\infty.
    \nonumber
        \end{equation}
        
        Then there exist a coefficient $c\in\PS(t_0,\Lambda_1,\Lambda_2,\DD,\mmu)$, a solution $u$ to (\ref{eqn:main}), and a positive real number $H_{13}$ for which (\ref{th:no-GEC}) holds true.
    \end{enumerate}

\end{cor}

We conclude with an example that shows that in the limit case $\beta=1/2$, and more generally in the case $\beta=1/m$ if we have assumptions of the form (\ref{hp:der-hir}), we can not eliminate the stabilization conditions (cfr.~\cite[Theorem~1.6]{2021-JMAA-Hirosawa}), even in the very special case where the operator $A$ is the identity. 

\begin{ex}[Failure of GEC without stabilization]\label{ex:no-way}
\begin{em}
    The function
    \begin{equation}
        w(t):=\sin(t)\exp\left(\frac{1}{8}\int_1^t\frac{1}{\tau}\sin^2(\tau)\,d\tau\right)
        \qquad
        \forall t\geq 1
        \label{defn:w-easy}
    \end{equation}
    is a solution to (\ref{eqn:ode}) for $t\geq 1$, with $\lambda:=1$ and
    \begin{equation}
        c(t):=1-\frac{\sin(2t)}{4t}+\frac{\sin^2(t)}{8t^2}-\frac{\sin^4(t)}{64t^2}
        \qquad
        \forall t\geq 1.
    \nonumber
    \end{equation}

    It is possible to check that 
    \begin{equation}
        |c^{(k)}(t)|\leq\frac{C_k}{t}
        \qquad
        \forall t\geq 1,
        \quad
        \forall k\geq 1,
    \nonumber
    \end{equation}
    and hence (\ref{hp:der-hir}) is true for every positive integer $m$, and in addition the integral of $|c(t)-1|$ grows as $\log t$. 
    
    On the other hand, the argument of the exponential in (\ref{defn:w-easy}) grows as $(\log t)/16$, and therefore the energy of the solution grows as $t^{1/8}$ and GEC fails.

    Starting from this example, one can construct a solution to (\ref{eqn:main}), or even to the wave equation (\ref{eqn:wave}) in any domain, with the same properties. 
\end{em}
\end{ex}

%%\clearpage

\setcounter{equation}{0}
\section{Estimates from above (proof of Theorem~\ref{thm:est-above})}\label{sec:positive}

Let us consider the family of ordinary differential equations (\ref{eqn:ode}). For every $\lambda>0$, and for every solution $\ul$, we introduce the \emph{Kowaleskian energy}
\begin{equation}
    \ekow(t):=\frac{\ul'(t)^2}{c_\infty^{1/2}}+\lambda^2 c_\infty^{1/2}\ul(t)^2,
    \nonumber
\end{equation}
where $c_\infty$ is the constant that appears in the stabilization assumption (\ref{hp:stab}). We should write more precisely $E_{\mathrm{kow},\lambda,u_\lambda}(t)$, in order to emphasize that this quantity depends also on $\lambda$ and on the solution $\ul$, but since we are dealing with one solution at each time we prefer to shorten the notation. 

It is well-known that solutions to (\ref{eqn:main}) with some coefficient $c(t)$ satisfy (\ref{th:mu-up}) if, for every $\lambda>0$, all solutions to (\ref{eqn:ode}) with the same coefficient $c(t)$ satisfy
\begin{equation}
    \ekow(t)\leq\ekow(t_0)\cdot H_1\exp\left(H_2 M(t)^{1/2}\right)
    \qquad
    \forall t\geq t_0
    \label{th:mu-up-ode}
\end{equation}
with the same constants $H_1$ and $H_2$, and similarly for (\ref{th:mu-down}).

In order to establish inequalities of this type, we exploit some classical energy estimates, that we describe below.

\subsection{Classical energy estimates}

\paragraph{\textmd{\textit{Kowaleskian estimate}}}

The time-derivative of the Kowaleskian energy is given by
\begin{equation}
    \ekow'(t)=-2\lambda^2\frac{c(t)-c_\infty}{c_\infty^{1/2}}\ul'(t) \ul(t),
    \nonumber
\end{equation}
from which we deduce the estimate
\begin{equation}
    \ekow'(t)\leq\lambda\frac{|c(t)-c_\infty|}{c_\infty^{1/2}}\ekow(t)
    \qquad
    \forall t\geq t_0,
    \nonumber
\end{equation}
and hence
\begin{equation}
    \ekow(b)\leq\ekow(a)\exp\left(\frac{\lambda}{\Lambda_1^{1/2}}\int_a^b|c(\tau)-c_\infty|\,d\tau\right)
    \label{est:kow}
\end{equation}
for every interval $[a,b]\subseteq[t_0,+\infty)$.

\paragraph{\textmd{\textit{Tarama estimate}}}

Following~\cite{2007-EJDE-Tarama}, let us introduce now the \emph{Tarama energy} (again we do not write the explicit dependence on $\lambda$ and $\ul$)
\begin{equation}
    \etar(t):=
    \frac{\ul'(t)^2}{c(t)^{1/2}}+\lambda^2 c(t)^{1/2}\ul(t)^2+\frac{1}{2}\frac{c'(t)}{c(t)^{3/2}}\ul'(t)\ul(t).
    \nonumber
\end{equation}

Let us consider any interval $[a,b]\subseteq[t_0,+\infty)$ such that
\begin{equation}
    \lambda\geq\frac{1}{2}\frac{|c'(t)|}{c(t)^{3/2}}
    \qquad
    \forall t\in[a,b].
    \label{hp:tar}
\end{equation}

In this interval it turns out that
\begin{equation}
    \left|\frac{1}{2}\frac{c'(t)}{c(t)^{3/2}}\ul'(t)\ul(t)\right|\leq
    \frac{1}{2}\frac{\ul'(t)^2}{c(t)^{1/2}}+
    \frac{1}{2}\lambda^2c(t)^{1/2}\ul(t)^2,
    \nonumber
\end{equation}
and hence
\begin{equation}
    \frac{1}{2}\left(
    \frac{\ul'(t)^2}{c(t)^{1/2}}+
    \lambda^2c(t)^{1/2}\ul(t)^2\right)\leq
    \etar(t)\leq\frac{3}{2}\left(
    \frac{\ul'(t)^2}{c(t)^{1/2}}+
    \lambda^2c(t)^{1/2}\ul(t)^2\right).
    \label{est:etar-1}
\end{equation}

Since we can write
\begin{equation}
    \frac{\ul'(t)^2}{c(t)^{1/2}}+
    \lambda^2c(t)^{1/2}\ul(t)^2=
    \frac{\ul'(t)^2}{c_\infty^{1/2}}
    \left(\frac{c_\infty}{c(t)}\right)^{1/2}+
    \lambda^2c_\infty^{1/2}\left(\frac{c(t)}{c_\infty}\right)^{1/2}\ul(t)^2,
    \nonumber
\end{equation}
and since
\begin{equation}
    \frac{\Lambda_1}{\Lambda_2}\leq
    \frac{c_\infty}{c(t)}\leq
    \frac{\Lambda_2}{\Lambda_1}
    \qquad\qquad\text{and}\qquad\qquad
    \frac{\Lambda_1}{\Lambda_2}\leq
    \frac{c(t)}{c_\infty}\leq
    \frac{\Lambda_2}{\Lambda_1},
    \nonumber
\end{equation}
from (\ref{est:etar-1}) we obtain that
\begin{equation}
    \frac{1}{2}\left(\frac{\Lambda_1}{\Lambda_2}\right)^{1/2}\ekow(t)\leq\etar(t)\leq \frac{3}{2}\left(\frac{\Lambda_2}{\Lambda_1}\right)^{1/2}\ekow(t)
    \qquad
    \forall t\geq t_0.
    \label{equi-ekow-etar}
\end{equation}

The time-derivative of the Tarama energy is given by
\begin{equation}
    \etar'(t)=
    \left(\frac{c''(t)}{2c(t)^{3/2}}-\frac{3}{4}\frac{c'(t)^2}{c(t)^{5/2}}\right)\ul'(t)\ul(t),
    \nonumber
\end{equation}
from which we deduce the estimate
\begin{eqnarray*}
    \etar'(t) & = &
    \frac{1}{4}\left(\frac{c''(t)}{c(t)^{3/2}}-\frac{3}{2}\frac{c'(t)^2}{c(t)^{5/2}}\right)\frac{1}{\lambda}\cdot 2\ul'(t)\cdot\lambda \ul(t)
    \\[0.5ex]
    & \leq &
    \frac{2\Lambda_1+3}{8\Lambda_1^{5/2}}\cdot\mmu(t)^2\frac{1}{\lambda}\cdot\left(
    \frac{\ul'(t)^2}{c(t)^{1/2}}+
    \lambda^2c(t)^{1/2}\ul(t)^2\right)
    \\[0.5ex]
    & \leq &
    \frac{2\Lambda_1+3}{4\Lambda_1^{5/2}}\cdot\frac{\mmu(t)^2}{\lambda}\cdot\etar(t)
\end{eqnarray*}
for every $t\in[a,b]$, and hence
\begin{equation}
    \etar(b)\leq\etar(a)\exp\left(\frac{2\Lambda_1+3}{4\Lambda_1^{5/2}}\cdot\frac{\MM(b)-\MM(a)}{\lambda}\right).
    \label{est:tar}
\end{equation}

\paragraph{\textmd{\textit{Mixed estimate}}}

Let $a$, $b$, $c$ be real numbers with $t_0\leq a<b<c$. Let us assume that (\ref{hp:tar}) holds true, so that we can use the Tarama energy in $[a,b]$. If we apply the Tarama estimate in $[a,b]$ we obtain (\ref{est:tar}), while from the Kowaleskian estimate in $[b,c]$ we obtain that
\begin{equation}
    \ekow(c)\leq \ekow(b)\exp\left(\frac{\lambda}{\Lambda_1^{1/2}}\int_b^c|c(\tau)-c_\infty|\,d\tau\right).    
    \nonumber
\end{equation}

If we take (\ref{equi-ekow-etar}) into account, by combining these two estimates we conclude that
\begin{eqnarray}
     \ekow(c) & \leq & 
     \frac{3\Lambda_2}{\Lambda_1} \ekow(a)\cdot
    \nonumber
     \\
     & & \mbox{}\cdot\exp\left(\frac{2\Lambda_1+3}{4\Lambda_1^{5/2}}\cdot\frac{\MM(b)-\MM(a)}{\lambda}+\frac{\lambda}{\Lambda_1^{1/2}}\int_b^c|c(\tau)-c_\infty|\,d\tau\right).
     \label{est:mix}
\end{eqnarray}

We call the latter the \emph{mixed estimate in $[a,c]$ with switch in $b$}. We observe that the same estimate is true also in the limit cases where either $b=a$ or $b=c$.

%%\clearpage

%%\clearpage

\subsection{Proof of statement~(1)}\label{sec:proof-up}

In this subsection we assume that $\mmu$ is non-decreasing and we prove that (\ref{th:mu-up-ode}) holds true, with $H_1$ and $H_2$ given by (\ref{defn:C12}), for every $\lambda>0$ and every solution $\ul$ to (\ref{eqn:ode}). The main tool is the following result, where the assumption on $\mmu$ is crucial. 

\begin{lemma}\label{lemma:mu-c}

    Let $t_0$, $\Lambda_1$, $\Lambda_2$, $\DD$ and $\mmu$ be as in Definition~\ref{defn:PS}. Let $\MM$ and $M$ be the functions defined in (\ref{defn:G}) and (\ref{defn:Max}), respectively.

    Let us assume in addition that $\mmu$ is non-decreasing.

    Then for every interval $[a,b]\subseteq[t_0,+\infty)$ it turns out that
    \begin{equation}
        \mmu(a)\int_a^b|c(\tau)-c_\infty|\,d\tau\leq
        2(\Lambda_2-\Lambda_1)^{1/2}
        M(b)^{1/2}.
        \label{th:lemma-kow}
    \end{equation}
\end{lemma}

\begin{proof}

Let $\tau_0\in(a,b)$ be such that
\begin{equation}
    \int_a^{\tau_0}|c(\tau)-c_\infty|\,d\tau=
    \int_{\tau_0}^b|c(\tau)-c_\infty|\,d\tau=
    \frac{1}{2}\int_a^b|c(\tau)-c_\infty|\,d\tau.
    \nonumber
\end{equation}

From the uniform hyperbolicity assumption (\ref{hp:hyp}), and the non-decreasing character of $\mmu(t)$, we deduce that
\begin{eqnarray*}
    \mmu(a)^2\int_a^{\tau_0}|c(\tau)-c_\infty|\,d\tau & \leq &
    (\Lambda_2-\Lambda_1)(\tau_0-a)\mmu(a)^2
    \\
    & \leq &
    (\Lambda_2-\Lambda_1)(\MM(\tau_0)-\MM(a))
    \\
    & \leq &
    (\Lambda_2-\Lambda_1)\MM(\tau_0).
\end{eqnarray*}

It follows that
\begin{eqnarray*}
    \mmu(a)\int_a^b|c(\tau)-c_\infty|\,d\tau & = &
    2\left[\mmu(a)^2\int_a^{\tau_0}|c(\tau)-c_\infty|\,d\tau\right]^{1/2}
    \left[\int_{\tau_0}^b|c(\tau)-c_\infty|\,d\tau\right]^{1/2}
    \\[1ex]
    & \leq &
    2(\Lambda_2-\Lambda_1)^{1/2}\MM(\tau_0)^{1/2}\DD(\tau_0)^{1/2},
\end{eqnarray*}
which implies (\ref{th:lemma-kow}).
\end{proof}

In order to prove (\ref{th:mu-up-ode}), we fix $t\geq t_0$, and we distinguish two cases according to the size of $\lambda$.

\paragraph{\textmd{\textit{Small frequencies}}}

Let us assume that
\begin{equation}
    \lambda\leq\frac{\mmu(t_0)}{2\Lambda_1^{3/2}}.
    \label{hp:lambda}
\end{equation}

In this case from Lemma~\ref{lemma:mu-c} applied in the interval $[a,b]:=[t_0,t]$ we deduce that
\begin{equation}
    \lambda\int_{t_0}^{t}|c(\tau)-c_\infty|\,d\tau\leq
    \frac{\mmu(t_0)}{2\Lambda_1^{3/2}}\int_{t_0}^{t}|c(\tau)-c_\infty|\,d\tau\leq
    \left(\frac{\Lambda_2-\Lambda_1}{\Lambda_1^3}\right)^{1/2}
    M(t)^{1/2}.
    \nonumber
\end{equation}

Therefore, from the Kowaleskian estimate (\ref{est:kow}) in $[t_0,t]$ we conclude that
\begin{equation}
    \ekow(t)\leq\ekow(t_0)\cdot
    \exp\left(\frac{(\Lambda_2-\Lambda_1)^{1/2}}{\Lambda_1^2}M(t)^{1/2}\right),
    \nonumber
\end{equation}
which proves (\ref{th:mu-up-ode}) in this case.

\paragraph{\textmd{\textit{Large frequencies}}}

Let us assume that (\ref{hp:lambda}) is false, and let us define
\begin{equation}
    t_1:=\max\left\{\tau\in[t_0,t]:2\Lambda_1^{3/2}\cdot\lambda\geq\mmu(\tau)\right\}.
    \nonumber
\end{equation}

We observe that the inequality in (\ref{hp:tar}) is satisfied for every $t\in[t_0,t_1]$, and hence we are allowed to use the Tarama energy in the interval $[t_0,t_1]$. Now we distinguish two cases.
\begin{itemize}
    \item Let us assume that
    \begin{equation}
        \MM(t_1)\leq
        \frac{4\Lambda_1^2}{2\Lambda_1+3}\cdot
        \lambda^2\int_{t_1}^{t}|c(\tau)-c_\infty|\,d\tau.
        \label{hp:t1}
    \end{equation}

    In this case we apply the mixed estimate (\ref{est:mix}) in the interval $[t_0,t]$ with switch in $t_1$, and we obtain that
    \begin{eqnarray}
        \ekow(t) & \leq &
        \frac{3\Lambda_2}{\Lambda_1}\ekow(t_0)\cdot
        \exp\left(\frac{2\Lambda_1+3}{4\Lambda_1^{5/2}}\cdot
        \frac{1}{\lambda}\MM(t_1)+\frac{\lambda}{\Lambda_1^{1/2}}\int_{t_1}^{t}|c(\tau)-c_\infty|\,d\tau\right)
        \nonumber
        \\
        & \leq &
        \frac{3\Lambda_2}{\Lambda_1}\ekow(t_0)\cdot
        \exp\left(
        \frac{2\lambda}{\Lambda_1^{1/2}}\int_{t_1}^{t}|c(\tau)-c_\infty|\,d\tau\right).
        \label{est:ekow-2.1}
    \end{eqnarray}

If $t_1=t$, then (\ref{est:ekow-2.1}) implies (\ref{th:mu-up-ode}). Otherwise, from the maximality of $t_1$ we deduce that $2\Lambda_1^{3/2}\cdot\lambda=\mmu(t_1)$.  Plugging this identity into (\ref{est:ekow-2.1}), from Lemma~\ref{lemma:mu-c} applied in the interval $[a,b]:=[t_1,t]$ we conclude that
\begin{equation}
    \frac{2\lambda}{\Lambda_1^{1/2}}\int_{t_1}^{t}|c(\tau)-c_\infty|\,d\tau=
    \frac{\mmu(t_1)}{\Lambda_1^2}\int_{t_1}^{t}|c(\tau)-c_\infty|\,d\tau\leq
    2\frac{(\Lambda_2-\Lambda_1)^{1/2}}{\Lambda_1^2}M(t)^{1/2},
    \nonumber
\end{equation}
which implies (\ref{th:mu-up-ode}) in this case.

\item  Let us assume that (\ref{hp:t1}) is false. In this case there exists $t_2\in[t_0,t_1)$ such that
    \begin{equation}
        \MM(t_2)=\frac{4\Lambda_1^2}{2\Lambda_1+3}\cdot
        \lambda^2\int_{t_2}^{t}|c(\tau)-c_\infty|\,d\tau.
        \label{hp:t2}
    \end{equation}

The existence of $t_2$ follows from a continuity argument because the difference between the left-hand side and the right-hand side is positive when $t_2=t_1$, and less than or equal to~0 when $t_2=t_0$. Again we are allowed to use the Tarama energy in the interval $[t_0,t_2]$, and therefore from the mixed estimate (\ref{est:mix}) in $[t_0,t]$ with switch in $t_2$ we obtain that
\begin{equation}
    \ekow(t)\leq\ekow(t_0)\cdot\frac{3\Lambda_2}{\Lambda_1}
    \exp\left(\frac{2\Lambda_1+3}{4\Lambda_1^{5/2}}\cdot
    \frac{\MM(t_2)}{\lambda}+\frac{\lambda}{\Lambda_1^{1/2}}\int_{t_2}^{t}|c(\tau)-c_\infty|\,d\tau\right).
    \nonumber
\end{equation}

Now from (\ref{hp:t2}) we conclude that the argument of the exponential is equal to
\begin{equation}
    \frac{(2\Lambda_1+3)^{1/2}}{\Lambda_1^{3/2}}[\MM(t_2)]^{1/2}\left[\int_{t_2}^{t}|c(\tau)-c_\infty|\,d\tau\right]^{1/2}\leq
    \frac{(2\Lambda_1+3)^{1/2}}{\Lambda_1^{3/2}}[\MM(t_2)\DD(t_2)]^{1/2},    
    \nonumber
\end{equation}
which implies (\ref{th:mu-up-ode}) also in this last case.
\qed

\end{itemize}

\subsection{Proof of statement~(2)}

In this subsection we consider the case where $\mmu$ is non-increasing. We prove that, for every $\lambda>0$, every solution $\ul$ to (\ref{eqn:ode}) satisfies
\begin{equation}
    \ekow(t)\leq\ekow(t_0)\cdot H_3\exp\left(H_4 M(t)^{1/2}\right)
    \qquad
    \forall t\geq t_0,
    \label{th:mu-down-ode}
\end{equation}
with $H_3$ and $H_4$ given by (\ref{defn:C34}). As already observed, this is enough to establish (\ref{th:mu-down}).

As in the previous case we fix $t>t_0$, and we distinguish small and large frequencies.

\paragraph{\textmd{\textit{Small frequencies}}}

If $\lambda$ satisfies (\ref{hp:lambda}), then from the Kowaleskian estimate (\ref{est:kow}) we obtain that
\begin{eqnarray*}
    \ekow(t) & \leq & \ekow(t_0)\cdot
    \exp\left(\frac{\lambda}{\Lambda_1^{1/2}}\int_{t_0}^t|c(\tau)-c_\infty|\,d\tau\right)
    \\
    & \leq &
    \ekow(t_0)\cdot\exp\left(\frac{\mmu(t_0)}{2\Lambda_1^2}\int_{t_0}^t|c(\tau)-c_\infty|\,d\tau\right)
    \\
    & \leq &
    \ekow(t_0)\cdot\exp\left(\frac{\mmu(t_0)}{2\Lambda_1^2}\DD(t_0)\right),
\end{eqnarray*}
which implies (\ref{th:mu-down-ode}) in this case.

\paragraph{\textmd{\textit{Large frequencies}}}

If (\ref{hp:lambda}) is false, then a fortiori $2\Lambda_1^{3/2}\cdot\lambda>\mmu(\tau)$ for every $\tau\in[t_0,t]$, due to the non-increasing character of $\mmu(t)$, and therefore we are allowed to use the Tarama energy wherever we want.

Let $t_1\in[t_0,t]$ be such that
\begin{equation}
    \MM(t_1)=\frac{4\Lambda_1^2}{2\Lambda_1+3}\cdot\lambda^2\int_{t_1}^t|c(\tau)-c_\infty|\,d\tau.
    \nonumber
\end{equation}

The existence of $t_1$ follows from a continuity argument because the difference between the left-hand side and the right-hand side is less than or equal to~0 when $t_1=t_0$, and greater than or equal to~0 when $t_1=t$. 

At this point from the mixed estimate (\ref{est:mix}) in $[t_0,t]$ with switch in $t_1$ we deduce that
\begin{equation}
    \ekow(t)\leq\ekow(t_0)\cdot\frac{3\Lambda_2}{\Lambda_1}\exp\left(
        \frac{2\Lambda_1+3}{4\Lambda_1^{5/2}}\cdot\frac{1}{\lambda}\MM(t_1)+\frac{\lambda}{\Lambda_1^{1/2}}\int_{t_1}^{t}|c(\tau)-c_\infty|\,d\tau\right),
    \nonumber
\end{equation}
and we conclude exactly as at the end of section~\ref{sec:proof-up}.
\qed

%%\clearpage

\setcounter{equation}{0}
\section{Counterexamples (proof of Theorem~\ref{thm:optimality})}\label{sec:negative}

The starting point in the construction of all counterexamples for wave equations with time-dependent propagation speed is the fundamental observation that the function
\begin{equation}
    w(t):=\frac{1}{m\lambda}\sin(m\lambda t)
    \exp\left(\frac{1}{8m^2}\int_{t_0}^t\ep(s)
    \sin^2(m\lambda s)\,ds\right)
    \label{defn:DGCS-w}
\end{equation}
is a solution to (\ref{eqn:ode}) with
\begin{equation}
    c(t):=m^2-
    \frac{\ep(t)}{4m\lambda}\sin(2m\lambda t)-
    \frac{\ep'(t)}{8m^2\lambda^2}\sin^2(m\lambda t)-
    \frac{\ep(t)^2}{64m^4\lambda^2}\sin^4(m\lambda t).
    \label{defn:DGCS-c}
\end{equation}

The key feature is that $w(t)$ grows exponentially with respect to time.

By choosing properly the frequency $\lambda$, the parameter $m$, and the function $\ep(t)$, we can obtain a coefficient $c(t)$ that respects our requirements in some interval $[a,b]$, and for which (\ref{eqn:ode}) admits a solution whose energy is equal to~1 at time $t=a$, and of order $\exp(M(b)^{1/2})$ (up to constants) at time $t=b$.

\begin{prop}[Activation step]\label{prop:main}
    In the same setting of Theorem~\ref{thm:optimality}, let us set
    \begin{equation}
        c_\infty:=\frac{\Lambda_1+\Lambda_2}{2}.
        \label{defn:c-infty}
    \end{equation}

    Then there exists a real number $H_{14}>0$ with the following property. For every triple of real numbers $A\geq t_0+1$, $L>0$, and $\Lambda\geq\lambda_0$, there exist three real numbers $a\in[A,A+1]$, $b>a$, $\lambda\geq\Lambda$, and a function $c:[a,b]\to(0,+\infty)$ of class $C^2$ such that
    \begin{enumerate}
    \renewcommand{\labelenumi}{(\arabic{enumi})}
        \item  it turns out that
        \begin{equation}
            % M(b)\geq 1
            % \qquad\text{and}\qquad
            M(b)\geq L\cdot M(A),
            \label{th:M(b)>>}
        \end{equation}

        \item the function $t\mapsto c(t)-c_\infty$ has compact support in $(a,b)$,

        \item $c$ satisfies the estimates
        \begin{gather}
            \Lambda_1\leq c(t)\leq\Lambda_2
            \qquad
            \forall t\in[a,b],
            \label{th:hyp-ab}
            \\[0.5ex]
            |c'(t)|\leq\mmu(t)
            \quad\text{and}\quad
            |c''(t)|\leq\mmu(t)^2
            \qquad
            \forall t\in[a,b],
            \label{th:der-ab}
            \\[0.5ex]
            \int_a^b|c(t)-c_\infty|\,d\tau\leq\frac{1}{2}\DD(b),
            \label{th:stab-ab}
        \end{gather}

        \item the solution $\ul$ to the ordinary differential equation (\ref{eqn:ode}) with initial data
        \begin{equation}
            \ul(a)=0,
            \qquad
            \ul'(a)=c_\infty^{1/4}
            \label{eqn:data}
        \end{equation}
        satisfies
        \begin{equation}
            \ul'(b)^2\geq \frac{c_\infty^{1/2}}{2}
            \exp\left(H_{14}M(b)^{1/2}\right).
            \label{th:ul'b}
        \end{equation}

    \end{enumerate}
\end{prop}

\subsection{Heuristics for Proposition~\ref{prop:main}}

Let us choose $a=A$, and $b\gg a$ so that $G(b)\gg G(a)$. Let us set
\begin{equation}
    m:=c_\infty^{1/2},
    \qquad
    \lambda:=\left[\frac{G(b)}{S(b)}\right]^{1/2},
    \qquad
    \ep(t):=\frac{\ep_0}{\lambda}\mmu(t)^2,
    \label{defn:heuristics}
\end{equation}
where $\ep_0>0$ is a small parameter. Let us define $c(t)$ as in (\ref{defn:DGCS-c}), and $\ul(t):=c_\infty^{1/4}w(t)$, where $w$ is defined by (\ref{defn:DGCS-w}).
Let us show that this choices fulfill to some extent the requirements of Proposition~\ref{prop:main}.
\begin{itemize}

    \item Condition (\ref{th:M(b)>>}) is automatic if we choose $b$ large enough.

    \item The hyperbolicity condition (\ref{th:hyp-ab}) follows if we choose $\ep_0$ small enough.

    \item Concerning the bounds (\ref{th:der-ab}) on derivatives, let us limit ourselves to second order derivatives. Now $c''(t)$ has many terms, but our ansatz is that the ``most dangerous one'' is the term in which we keep $\ep(t)$ and we derive $\sin(2m\lambda t)$ twice. We end up with
    \begin{equation}
        \ep(t)\cdot m\lambda\sin(2m\lambda t)=
        \ep_0 m \sin(2m\lambda t)\mmu(t)^2,
    \nonumber
    \end{equation}
    and this can be controlled by $\mmu(t)^2$ by choosing $\ep_0$ properly.

    \item Concerning the stabilization condition (\ref{th:stab-ab}), let us pretend for a while that in (\ref{defn:DGCS-c}) there are only the first two terms. In this case we obtain that
    \begin{equation}
        \int_a^b|c(\tau)-c_\infty|\,d\tau\leq
        \int_a^b\frac{\ep(t)}{4m\lambda}\,d\tau=
        \frac{\ep_0}{4m\lambda^2}\int_a^b\mmu(\tau)^2\,d\tau\leq
        \frac{\ep_0}{4m}\frac{S(b)}{G(b)}G(b),
    \nonumber
    \end{equation}
    and therefore again (\ref{th:stab-ab}) is satisfied if we choose $\ep_0$ small enough.

    \item The (Kowaleskian) energy of $\ul$ is equal to~1 at time $t=a$, and in $t=b$ it is proportional to the exponential of
    \begin{multline}
        \int_a^b\ep(\tau)\sin^2(m\lambda\tau)\,d\tau\sim
        \frac{1}{2}\int_a^b\ep(\tau)\,d\tau=
        \frac{\ep_0}{2\lambda}\int_a^b\mmu(\tau)^2\,d\tau
        \\[0.5ex]
        =
        \frac{\ep_0}{2}\left[\frac{S(b)}{G(b)}\right]^{1/2}(G(b)-G(a))\sim
        \frac{\ep_0}{2}[S(b)G(b)]^{1/2}\sim
        \frac{\ep_0}{2}M(b)^{1/2},
        \label{est:exp-naive}
    \end{multline}
    where in the first step we approximated the oscillating integral by replacing the function $\sin^2(m\lambda t)$ by its average, and then we used also that $G(b)\gg G(a)$.
\end{itemize}

Certainly, there are several gaps in the previous arguments. Let's identify these gaps and outline how they will be addressed in the full proof.
\begin{itemize}
    \item If $\ep(t)$ is defined as above, then $c(t)-c_\infty$ does not have compact support within $(a,b)$. This can be fixed by using a cutoff function near the boundary.

    \item In the estimates of derivatives, and in the stabilization condition, one has to consider all the terms in the definition of $c(t)$. With some patience this can be done. Incidentally, this is the point where we leverage assumption (\ref{hp:gamma-C3}), which prevents the function $\mmu$ from growing too rapidly.

    \item The approximation of the oscillating integral needs a rigorous justification.

    \item When considering the solution $\ul(t)$, we need to know that $\sin(m\lambda a)=\sin(m\lambda b)=0$. This necessitates careful selection of $a$ and $b$ to satisfy these conditions.

    \item In the final step of (\ref{est:exp-naive}), it is beneficial to have $S(b)G(b)=M(b)$. Although this equality does not hold for all values of $b$, it does hold for sufficiently many, requiring careful selection of $b$.
\end{itemize}

%%\clearpage

\subsection{Technical preliminaries}

In this subsection, we establish three technical results aimed at defining three constants, $\Gamma_1$, $\Gamma_2$, $\Gamma_3$, which are crucial for selecting parameters in the proof of  Proposition~\ref{prop:main}. Here we work with a generic non-negative function $g(t)$, and its integral
\begin{equation}
    G(t):=\int_{t_0}^t g(\tau)\,d\tau
    \qquad
    \forall t\geq t_0.
    \label{defn:G-bis}
\end{equation}

In the sequel of the paper we apply these results to $g(t):=\mmu(t)^2$, and therefore definition (\ref{defn:G-bis}) is coherent with (\ref{defn:G}).

The first result says, in a nutshell, that if $g(t)$ grows at most exponentially, then also $G(t)$ grows at most exponentially, and in particular it is possible to estimate $G(t+1)$ in terms of $G(t)$.

\begin{lemma}\label{lemma:G<G}
    Let $t_0$ be a real number, let $g:[t_0,+\infty)\to[0,+\infty)$ be a monotone (either non-increasing or non-decreasing) function of class $C^1$, and let $G$ be defined by (\ref{defn:G-bis}). 

    Let us assume that there exists a constant $\Lambda_4$ such that
    \begin{equation}
        |g'(t)|\leq\Lambda_4 g(t)
        \qquad
        \forall t\geq t_0.
        \label{hp:g'-G}
    \end{equation}

    Then there exist two constants $\Gamma_1$ and $\Gamma_2$ such that
    \begin{equation}
        g(t)\leq\Gamma_1 G(t)
        \qquad
        \forall t\geq t_0+1,
        \label{th:lemma1-g<G}
    \end{equation}
    and
    \begin{equation}
        G(t+1)\leq\Gamma_2 G(t)
        \qquad
        \forall t\geq t_0+1.
        \label{th:lemma1-G<G}
    \end{equation}
\end{lemma}

\begin{proof}

Let us assume for a while that $G(t_0+1)>0$. In this case we set 
\begin{equation}
    \Gamma_1:=\max\left\{\Lambda_4,\frac{g(t_0+1)}{G(t_0+1)}\right\},
    \qquad\qquad
    \Gamma_2:=\exp(\Gamma_1),
    \nonumber
\end{equation}
and we consider the function 
\begin{equation}
    D(t):=\Gamma_1 G(t)-g(t)
    \qquad
    \forall t\geq t_0.
    \nonumber
\end{equation}

From the definition of $\Gamma_1$, and assumption (\ref{hp:g'-G}), we deduce that $D(t)$ is non-decreasing in $[t_0,+\infty)$, and non-negative for $t=t_0+1$. It follows that $D(t)\geq 0$ for every $t\geq t_0+1$, which is equivalent to (\ref{th:lemma1-g<G}).

At this point, we can interpret (\ref{th:lemma1-g<G}) as a differential inequality for the function $G(t)$. Integrating this differential inequality we conclude that
\begin{equation}
    G(t_2)\leq G(t_1)\exp(\Gamma_1(t_2-t_1))
    % \qquad
    % \forall t_2\geq t_1\geq t_0+1.
    \nonumber
\end{equation}
for every interval $[t_1,t_2]\subseteq[t_0+1,+\infty)$. Choosing $t_1:=t$ and $t_2:=t+1$ in this inequality, we obtain (\ref{th:lemma1-G<G}).

In the case where $G(t_0+1)=0$, we observe that this condition implies that necessarily $g(t)\equiv 0$ for every $t\in[t_0,t_0+1]$, and therefore the same argument works with $\Gamma_1:=\Lambda_4$, because in this case $D(t_0+1)=0$ for free.    
\end{proof}

%%\clearpage

The second result is a rather classical estimate for an oscillating integral.

\begin{lemma}\label{lemma:osc-int}
    Let $t_0$, $g$ and $G$ be as in Lemma~\ref{lemma:G<G}.

    Then for every interval $[\alpha,\beta]\subseteq[t_0,+\infty)$, and for every real number $\ell>0$, it turns out that
    \begin{equation}
        \int_\alpha^\beta g(\tau)\sin^2(\ell\tau)\,d\tau\geq
        \frac{1}{2}(G(\beta)-G(\alpha))-
        \frac{1}{\ell}\max\{g(\alpha),g(\beta)\}.
    \nonumber
    \end{equation}
\end{lemma}

\begin{proof}

To begin with, we observe that
\begin{equation}
    \int_\alpha^\beta g(\tau)\sin^2(\ell\tau)\,d\tau=
    \frac{1}{2}\int_\alpha^\beta g(\tau)\,d\tau-
    \frac{1}{2}\int_\alpha^\beta g(\tau)\cos(2\ell\tau)\,d\tau.
    \nonumber
\end{equation}

The first term is equal to $(G(\beta)-G(\alpha))/2$. In the second term we integrate by parts and we obtain that
\begin{equation}
    \int_\alpha^\beta g(\tau)\cos(2\ell\tau)\,d\tau=
    \frac{1}{2\ell}[g(\beta)\sin(2\ell \beta)-g(\alpha)\sin(2\ell \alpha)]-
    \frac{1}{2\ell}\int_\alpha^\beta g'(\tau)\sin(2\ell\tau)\,d\tau,
    \nonumber
\end{equation}
and hence
\begin{equation}
    \left|\int_\alpha^\beta g(\tau)\cos(2\ell\tau)\,d\tau\right|\leq
    \frac{1}{\ell}\max\{g(\alpha),g(\beta)\}+
    \frac{1}{2\ell}\int_\alpha^\beta |g'(\tau)|\,d\tau.
    \nonumber
\end{equation}

Finally, by the monotonicity of $g$ we conclude that
\begin{equation}
    \int_\alpha^\beta |g'(\tau)|\,d\tau=
    |g(\beta)-g(\alpha)|\leq
    \max\{g(\alpha),g(\beta)\},
    \nonumber
\end{equation}
which completes the proof.    
\end{proof}

%%\clearpage

In the third technical lemma we consider a function of the form (\ref{defn:DGCS-c}), where $\ep(t)$ is the product of two functions $\theta(t)$ and $g(t)$, up to multiplicative constants. We show that suitable estimates on the functions $\theta(t)$ and $g(t)$, and their derivatives, yield corresponding estimates on $c(t)$ and its derivatives. 

\begin{lemma}\label{lemma:computation}
    Let $[a,b]\subseteq\re$ be an interval, let $\lambda\geq 1$ be a real number, and let $\theta:[a,b]\to[0,+\infty)$ and $g:[a,b]\to[0,+\infty)$ be two functions of class $C^3$.

    Let us assume that there exist two constants $\Lambda_5$ and $\Lambda_6$ such that, for every $t\in[a,b]$, the following estimates hold:
    \begin{gather}
        \max\left\{|\theta(t)|,|\theta'(t)|,|\theta''(t)|,|\theta'''(t)|\right\}\leq\Lambda_5,
        \label{hp:lemma3-theta}
        \\[1ex]
        \max\left\{|g'(t)|,|g''(t)|,|g'''(t)|\right\}\leq \Lambda_6 g(t),
        \label{hp:lemma3-g}
        \\[0.5ex]
        g(t)\leq\lambda^2.
        \label{hp:lemma3-lambda}
    \end{gather}

    Given two real numbers $m>0$ and $\ep_0\in(0,1]$, we consider the function $c(t)$ defined by (\ref{defn:DGCS-c}) with
    \begin{equation}
        \ep(t):=\frac{\ep_0}{\lambda}\theta(t)g(t)
        \qquad
        \forall t\in[a,b].
        \label{defn:ep-theta-g}
    \end{equation}

    Then there exists a constant $\Gamma_3$, which depends only on $m$, $\Lambda_5$, $\Lambda_6$, such that
    \begin{equation}
        \max\left\{\lambda^2|c(t)-m^2|,\lambda|c'(t)|,|c''(t)|\right\}\leq
        \Gamma_3\ep_0g(t)
        \qquad
        \forall t\in[a,b].
    \label{th:c-c'-c''}
    \end{equation}
\end{lemma}

\begin{proof}
    With some patience, from (\ref{defn:DGCS-c}) we obtain that there exists a constant $\Gamma(m)$, that depends only on $m$, such that (for the sake of shortness, here we omit the dependence on $t$ in the right-hand sides)
    \begin{eqnarray*}
        |c(t)-m^2| & \leq &
        \Gamma(m)\left\{\frac{|\ep|}{\lambda}+\frac{|\ep'|+\ep^2}{\lambda^2}\right\},
        \\
        |c'(t)| & \leq &
        \Gamma(m)\left\{|\ep|+\frac{|\ep'|+\ep^2}{\lambda}+\frac{|\ep \ep'|+|\ep''|}{\lambda^2}\right\},
        \\
        |c''(t)| & \leq &
        \Gamma(m)\left\{\lambda|\ep|+
        |\ep'|+\ep^2+
        \frac{|\ep \ep'|+|\ep''|}{\lambda}+
        \frac{|\ep'|^2+|\ep \ep''|+|\ep'''|}{\lambda^2}\right\}.
    \end{eqnarray*}

    On the other hand, from (\ref{defn:ep-theta-g}), (\ref{hp:lemma3-theta}) and (\ref{hp:lemma3-g}) we deduce that
    \begin{equation}
        \begin{array}{c@{\qquad\quad}c}
        |\ep(t)|\leq\dfrac{\ep_0}{\lambda}\Lambda_5 g(t),     &  |\ep'(t)|\leq\dfrac{\ep_0}{\lambda}\Lambda_5(1+\Lambda_6) g(t),
        \\[3ex]
        |\ep''(t)|\leq\dfrac{\ep_0}{\lambda}\Lambda_5(1+3\Lambda_6) g(t),
        & 
        |\ep'''(t)|\leq\dfrac{\ep_0}{\lambda}\Lambda_5(1+7\Lambda_6) g(t).
        \end{array}
    \nonumber
    \end{equation}

    Taking into account that $\lambda\geq 1$, and assumption (\ref{hp:lemma3-lambda}), from all these estimates we deduce (\ref{th:c-c'-c''}).
\end{proof}

%%\clearpage

\subsection{Proof of Proposition~\ref{prop:main}}

Since the statement is invariant by time-translations, we assume, without loss of generality, that $t_0\geq 0$. We set also, throughout this proof, $m:=c_\infty^{1/2}$ and $g(t):=\mmu(t)^2$.

\paragraph{\textmd{\textit{Choice of $b$}}}

Let $\Lambda_3$ be the constant that appears in our assumption (\ref{hp:gamma-C3}), and let $\Gamma_1$ and $\Gamma_2$ be the two constants of Lemma~\ref{lemma:G<G} corresponding to $\Lambda_4:=2\Lambda_3$.

Let us choose a real number $b$ with the following properties:
\begin{gather}
    b\geq A+4,
    \qquad\qquad
    M(b)\geq L\cdot M(A),
    \qquad\qquad
    % M(b)\geq 1,
    % \qquad
    4\Gamma_1\cdot S(b)\leq 1,
    \label{hp:b-1}
    \\[0.5ex]
    G(b)\geq 2\Gamma_2 G(A+2),
    \qquad\qquad
    \frac{G(b)}{S(b)}\geq\max\left\{\frac{16\pi^2}{m^2},4,4\lambda_0^2,4\Lambda^2\right\},
    \label{hp:b-2}
    \\[1ex]
    M(b)=G(b)S(b).
    \label{hp:b-3}
\end{gather}

The choice is possible because all the inequalities in (\ref{hp:b-1}) and (\ref{hp:b-2}) are true provided that $b$ is large enough, and there exists an unbounded set $\mathcal{B}\subseteq[t_0,+\infty)$ such that (\ref{hp:b-3}) is satisfies for every $b\in\mathcal{B}$. Indeed, for every $r>0$ the number
\begin{equation}
    \min\{t\geq t_0:G(t)S(t)\geq r\}
    \nonumber
\end{equation}
belongs to $\mathcal{B}$.

\paragraph{\textmd{\textit{Choice of $\lambda$}}}

Let us set
\begin{equation}
    \lambda:=\frac{2\pi}{mb}\left\lfloor\left[\frac{G(b)}{S(b)}\right]^{1/2}\cdot\frac{mb}{2\pi}\right\rfloor,
    \nonumber
\end{equation}
where $\lfloor\alpha\rfloor$ denotes the floor function, namely the greatest integer less than or equal to $\alpha$. We observe that this definition is asymptotically equivalent to the definition of $\lambda$ in (\ref{defn:heuristics}), but with the extra property that $m\lambda b$ is an integer multiple of $2\pi$.

We claim that
\begin{equation}
    \frac{1}{2}\left[\frac{G(b)}{S(b)}\right]^{1/2}\leq
    \lambda\leq
    \left[\frac{G(b)}{S(b)}\right]^{1/2},
    \label{est:lambda=S/G}
\end{equation}
and as a consequence
\begin{equation}
    \lambda\geq
    \max\left\{\frac{2\pi}{m},1,\lambda_0,\Lambda\right\},
    \label{est:lambda>max}
\end{equation}
and
\begin{equation}
    \lambda^2\geq\mmu(t)^2
    \qquad
    \forall t\in[A,b].
    \label{est:lambda>gamma}
\end{equation}

Indeed, the estimate from above in (\ref{est:lambda=S/G}) is immediate from the definition of integer part. As for the estimate from below, we observe that for every pair of real numbers $X\geq 4\pi/m$ and $b\geq 1$ it turns out that
\begin{equation}
    \frac{2\pi}{mb}\left\lfloor X\cdot\frac{mb}{2\pi}\right\rfloor\geq
    \frac{2\pi}{mb}\left( X\cdot\frac{mb}{2\pi}-1\right)=
    X-\frac{2\pi}{mb}\geq\frac{X}{2}.
    \nonumber
\end{equation}

Applying this inequality with $X:=[G(b)/S(b)]^{1/2}$, we obtain the estimate from below in (\ref{est:lambda=S/G}). At this point, (\ref{est:lambda>max}) follows from the second condition in (\ref{hp:b-2}), while from (\ref{th:lemma1-g<G}) (in this point we need that $A\geq t_0+1$) and the last condition in (\ref{hp:b-1}) we deduce that
\begin{equation}
    \lambda^2\geq
    \frac{1}{4}\frac{G(b)}{S(b)}\geq
    \frac{1}{4}\frac{G(t)}{S(b)}\geq
    \frac{1}{4\Gamma_1}\frac{\mmu(t)^2}{S(b)}\geq
    \mmu(t)^2
    \qquad
    \forall t\in[A,b],
    \nonumber
\end{equation}
which is exactly (\ref{est:lambda>gamma}).

\paragraph{\textmd{\textit{Choice of $a$}}}

Since $m\lambda\geq 2\pi$, there exists an integer $k_0$ such that
\begin{equation}
    a:=\frac{2\pi}{m\lambda}\cdot k_0\in[A,A+1].
    \label{defn:a}
\end{equation}

\paragraph{\textmd{\textit{Choice of $c(t)$}}}

From (\ref{defn:a}) and the first condition in (\ref{hp:b-1}) we know that the length of $[a,b]$ is at least~3. Let us choose a cutoff function $\theta$ of class $C^2$ with compact support in $(a,b)$ such that $\theta(t)=1$ for every $t\in[a+1,b-1]$, and
\begin{equation}
    \max\left\{|\theta(t)|,|\theta'(t)|,|\theta''(t)|,|\theta'''(t)|\right\}\leq 100
    \qquad
    \forall t\in(a,b)
    \nonumber
\end{equation}
(of course 100 can be replaced by a smaller constant, but the point here is that this constant can be chosen to be independent of $[a,b]$). Let $\Gamma_3$ be the constant of Lemma~\ref{lemma:computation} corresponding to $\Lambda_5:=100$ and $\Lambda_6:=(6\Lambda_3^2 + 2 \Lambda_3) $. Let us choose
\begin{equation}
    \ep_0:=\min\left\{\frac{\Lambda_2-\Lambda_1}{2\Gamma_3},\frac{1}{8\Gamma_3},4m^3\log 2\right\},
    \label{hp:ep0}
\end{equation}
and let us set
\begin{equation}
    \ep(t):=\frac{\ep_0}{\lambda}\theta(t)\mmu(t)^2
    \qquad
    \forall t\in[a,b].
    \label{defn:ept}
\end{equation}

We observe that this definition coincides with the definition of $\ep(t)$ given in (\ref{defn:heuristics}), with the addition of $\theta(t)$ in order to guarantee the compact support. Finally, we define $c(t)$ as in (\ref{defn:DGCS-c}).

\paragraph{\textmd{\textit{Compact support and estimates on $c(t)$}}}

As already observed, the function $c(t)-c_\infty$ has compact support in $(a,b)$ due to the presence of $\theta(t)$.

Concerning the uniform hyperbolicity, from (\ref{th:c-c'-c''}) we know that
\begin{equation}
    |c(t)-c_\infty|\leq\frac{\Gamma_3\ep_0\mmu(t)^2}{\lambda^2}
    \qquad
    \forall t\in[a,b],
    \label{est:c-m}
\end{equation}
and hence from (\ref{est:lambda>gamma}) and the first condition in (\ref{hp:ep0}) we deduce that
\begin{equation}
    |c(t)-c_\infty|\leq
    \Gamma_3\ep_0\leq
    \frac{\Lambda_2-\Lambda_1}{2}
    \qquad
    \forall t\in[a,b].
    \nonumber
\end{equation}

Recalling that $c_\infty=(\Lambda_1+\Lambda_2)/2$, this inequality implies (\ref{th:hyp-ab}).

Concerning estimates on derivatives, from (\ref{th:c-c'-c''}), (\ref{est:lambda>gamma}), and the second condition in (\ref{hp:ep0}), we obtain that
\begin{equation}
    |c'(t)|\leq
    \frac{\Gamma_3\ep_0}{\lambda}\cdot\mmu(t)^2=
    \Gamma_3\ep_0\cdot\frac{\mmu(t)}{\lambda}\cdot\mmu(t)\leq
    \Gamma_3\ep_0\cdot\mmu(t)\leq
    \frac{1}{8}\mmu(t),
    \nonumber
\end{equation}
and
\begin{equation}
    |c''(t)|\leq
    \Gamma_3\ep_0\cdot\mmu(t)^2\leq
    \frac{1}{8}\mmu(t)^2,
    \nonumber
\end{equation}
which prove (\ref{th:der-ab}).

Concerning the stabilization property, we integrate (\ref{est:c-m}) over $[a,b]$, and we exploit the estimate from below in (\ref{est:lambda=S/G}). We obtain that
\begin{equation}
    \int_a^b|c(\tau)-c_\infty|\,d\tau\leq
    \frac{\Gamma_3\ep_0}{\lambda^2}(G(b)-G(a))\leq
    \Gamma_3\ep_0\cdot 4\frac{S(b)}{G(b)}\cdot G(b)\leq
    4\Gamma_3\ep_0 S(b),
    \nonumber
\end{equation}
and therefore (\ref{th:stab-ab}) follows from the second condition in (\ref{hp:ep0}).

\paragraph{\textmd{\textit{Estimates on $\ul(t)$}}}

Let us consider the solution $\ul(t)$ to equation (\ref{eqn:ode}) with initial data (\ref{eqn:data}). Since $m\lambda a$ is an integer multiple of $2\pi$, one can check that
\begin{equation}
    \ul(t)=c_\infty^{1/4}w(t)
    \qquad
    \forall t\in[a,b],
    \nonumber
\end{equation}
with $w$ given by (\ref{defn:DGCS-w}). Since also $m\lambda b$ is a multiple of $2\pi$, from the same explicit formula we deduce that
\begin{equation}
    \ul'(b)^2=c_\infty^{1/2}\exp\left(
    \frac{1}{4c_\infty}\int_a^b\ep(\tau)\sin^2(m\lambda\tau)\,d\tau\right).
    \label{eqn:ul'b}
\end{equation}

Let us estimate the argument of the exponential. Since $\theta(t)\equiv 1$ in the interval $[a+1,b-1]$, from (\ref{defn:ept}) we obtain that
\begin{equation}
    I:=\int_a^b\ep(\tau)\sin^2(m\lambda\tau)\,d\tau\geq
    \int_{a+1}^{b-1}\ep(\tau)\sin^2(m\lambda\tau)\,d\tau=
    \frac{\ep_0}{\lambda}\int_{a+1}^{b-1}\mmu(\tau)^2\sin^2(m\lambda\tau)\,d\tau.
    \nonumber
\end{equation}

The last oscillating integral can be estimated by applying Lemma~\ref{lemma:osc-int} in the interval $[a+1,b-1]$ with $\ell:=m\lambda$. We obtain that
\begin{equation}
    I\geq\frac{\ep_0}{2\lambda}(G(b-1)-G(a+1))-
    \frac{\ep_0}{m\lambda^2}\max\left\{\mmu(b-1)^2,\mmu(a+1)^2\right\}.
    \nonumber
\end{equation}

If we estimate the first $\lambda$ from above as in (\ref{est:lambda=S/G}), and the second $\lambda$ from below as in (\ref{est:lambda>gamma}), we deduce that
\begin{equation}
    I\geq\frac{\ep_0}{2}\left[\frac{S(b)}{G(b)}\right]^{1/2}(G(b-1)-G(a+1))-\frac{\ep_0}{m}.
    \nonumber
\end{equation}

Finally, from (\ref{th:lemma1-G<G}) and the first condition in (\ref{hp:b-2}) we obtain that
\begin{equation}
    G(b-1)\geq\frac{1}{\Gamma_2}G(b)
    \qquad\text{and}\qquad
    G(a+1)\leq G(A+2)\leq
    \frac{1}{2\Gamma_2}G(b),
    \nonumber
\end{equation}
and hence
\begin{equation}
    I\geq
    \frac{\ep_0}{2}\left[\frac{S(b)}{G(b)}\right]^{1/2}\cdot
    \frac{1}{2\Gamma_2}G(b)-\frac{\ep_0}{m}=
    \frac{\ep_0}{4\Gamma_2}[S(b)G(b)]^{1/2}-\frac{\ep_0}{m}.
    \nonumber
\end{equation}

Recalling (\ref{hp:b-3}) and the last condition in (\ref{hp:ep0}), we conclude that
\begin{equation}
    I\geq
    \frac{\ep_0}{4\Gamma_2}M(b)^{1/2}-4m^2\log 2.
    \nonumber
\end{equation}

Plugging this inequality into (\ref{eqn:ul'b}) we obtain (\ref{th:ul'b}) with $H_{14}:=\ep_0/(16c_\infty\Gamma_2)$. We observe that $\ep_0$ depends only on $\Lambda_1$, $\Lambda_2$, $\Lambda_3$, and $\Gamma_2$ depends also on the behavior of $\gamma(t)$ in the interval $[t_0,t_0+1]$.
\qed

%%\clearpage

\subsection{From Proposition~\ref{prop:main} to Theorem~\ref{thm:optimality}}\label{sec:iteration}

\paragraph{\textmd{\textit{Iterative construction}}}

For every integer $k\geq 1$, we define a real number $A_k$, an interval $[a_k,b_k]\subseteq[t_0,+\infty)$, a real number $\lambda_k\geq \lambda_0$, and a function $c_k:[a_k,b_k]\to(0,+\infty)$ in the following way.

To begin with, set $A_1:=t_0+1$, and we apply Proposition~\ref{prop:main} with $A:=A_1$, $L:=4$ and $\Lambda:=\lambda_0$. In this way we obtain the first interval $[a_1,b_1]$, the first number $\lambda_1$, and the first function $c_1:[a_1,b_1]\to(0,+\infty)$. Then we proceed by induction. Let us assume that for some $k\geq 1$ we have already defined $[a_k,b_k]$, $\lambda_k$ and $c_k$. Then we choose a real number $A_{k+1}$ such that
\begin{equation}
    A_{k+1}\geq b_k
    \qquad\text{and}\qquad
    \DD(A_{k+1})\leq\frac{1}{2}\DD(b_k),
    \label{defn:Ak+1}
\end{equation}
and we apply Proposition~\ref{prop:main} with $A:=A_{k+1}$, $L:=4(k+1)^2$ and $\Lambda:=\lambda_k+1$. We obtain the new interval $[a_{k+1},b_{k+1}]$, the new real number $\lambda_{k+1}$, and the new function $c_{k+1}:[a_{k+1},b_{k+1}]\to(0,+\infty)$.

We observe that
\begin{itemize}
    \item the intervals $[a_k,b_k]$ are essentially disjoint (no interior points in common),

    \item $M(b_k)\geq 4k^2 M(A_k)$ for every $k\geq 1$.

    \item the sequences $\{A_k\}$, $\{b_k\}$ and  $\{\lambda_k\}$ are increasing, and tend to $+\infty$.
\end{itemize}

\paragraph{\textmd{\textit{Construction of the propagation speed}}}

Let us define $c_\infty$ as in (\ref{defn:c-infty}), and let us consider the function $c:[t_0,+\infty)\to(0,+\infty)$ defined by
\begin{equation}
    c(t):=\begin{cases}
        c_k(t) & \text{if $t\in[a_k,b_k]$ for some $k\geq 1$},
        \\
        c_\infty & \text{otherwise}.
    \end{cases}
    \nonumber
\end{equation}

The function $c(t)$ is well-defined and of class $C^2$ because the intervals $(a_k,b_k)$ are disjoint, and each function $t\mapsto c_k(t)-c_\infty$ has compact support in $(a_k,b_k)$.

Let us check that $c\in\PS(t_0,\Lambda_1,\Lambda_2,S,\mmu)$. The uniform hyperbolicity (\ref{hp:hyp}), and the control on derivatives (\ref{hp:der}), are satisfied because they are true in each interval $[a_k,b_k]$ due to (\ref{th:hyp-ab}) and (\ref{th:der-ab}), and outside these intervals the function is constant. 

In order to check the stabilization at infinity (\ref{hp:stab}), we observe that every $t\geq t_0$ belongs either to $[a_k,b_k]$ for some $k\geq 1$, or to $[b_{k-1},a_k]$ for some $k\geq 2$, or to $[t_0,a_1]$.

Let us start by considering the case where $t\in[a_k,b_k]$ for some $k\geq 1$. Due to the monotonicity of $\DD$, from (\ref{defn:Ak+1}) we deduce that $\DD(b_{k+1})\leq\DD(A_{k+1})\leq\DD(b_k)/2$, and therefore by induction
\begin{equation}
    \DD(b_{k+i})\leq\frac{1}{2^i}\DD(b_k)
    \qquad
    \forall k\geq 1,
    \quad
    \forall i\geq 0.
    \nonumber
\end{equation}

Thus from (\ref{th:stab-ab}) we conclude that
\begin{multline*}
    \qquad
    \int_t^{+\infty}|c(\tau)-c_\infty|\,d\tau\leq
    \int_{a_k}^{+\infty}|c(\tau)-c_\infty|\,d\tau=
    \sum_{i=0}^{\infty}\int_{a_{k+i}}^{b_{k+i}}|c(\tau)-c_\infty|\,d\tau
    \\[1ex]
    \leq\sum_{i=0}^{\infty}\frac{1}{2}\DD(b_{k+i})\leq
    \frac{1}{2}\sum_{i=0}^{\infty}\frac{1}{2^{i}}\DD(b_{k})=
    \DD(b_k)\leq
    \DD(t).
    \qquad
\end{multline*}

If $t\in[b_{k-1},a_{k}]$ for some $k\geq 2$, then we observe that $c(t)=c_\infty$ in this interval, and hence
\begin{equation}
    \int_t^{+\infty}|c(\tau)-c_\infty|\,d\tau=
    \int_{a_k}^{+\infty}|c(\tau)-c_\infty|\,d\tau\leq
    S(b_k)\leq
    S(t),
    \nonumber
\end{equation}
and analogously if $t\in[t_0,a_1]$.

\paragraph{\textmd{\textit{Construction of the solution}}}

For every integer $k\geq 1$, we consider the solution $v_k:[t_0,+\infty)\to\re$ to the ordinary differential equation
\begin{equation}
    v_k''(t)+\lambda_k^2 c(t)v_k(t)=0,
    \nonumber
\end{equation}
with ``initial'' data
\begin{equation}
    v_k(a_k)=0,
    \qquad
    v_k'(a_k)=c_\infty^{1/4}.
    \label{eqn:vk-data}
\end{equation}

Let
\begin{equation}
    E_k(t):=\frac{v_k'(t)^2}{c_\infty^{1/2}}+\lambda_k^2 c_\infty^{1/2} v_k(t)^2
    \nonumber
\end{equation}
denote the (Kowaleskian) energy of this solution. From (\ref{eqn:vk-data}) we know that $E_k(a_k)=1$, while from statement~(4) of Proposition~\ref{prop:main} we know that
\begin{equation}
    E_k(b_k)\geq\frac{v_k'(b_k)^2}{c_\infty^{1/2}}\geq
    \frac{1}{2}\exp\left(H_{14}M(b_k)^{1/2}\right).
    \label{est:Ekbk}
\end{equation}

We claim that there exists two constants $H_{15}$ and $H_{16}$, both independent of $k$, such that
\begin{equation}
    0<E_k(t_0)\leq H_{15}\exp\left(H_{16}M(A_k)^{1/2}\right).
    \label{est:Ekt0}
\end{equation}

Indeed, the estimate from below is trivial because otherwise $v_k$ would be identically zero. As for the estimate from above, in the interval $[t_0,A_k]$ we can apply the estimates from below corresponding to Theorem~\ref{thm:est-above} (see Remark~\ref{rmk:below}), which read as
\begin{equation}
    1=E_k(A_k)\geq
    E_k(t_0)\cdot (H_{17})^{-1}\exp(-H_{18} M(A_k)^{1/2}),
    \nonumber
\end{equation}
where the two constants $H_{17}$ and $H_{18}$ are either equal to $H_1$ and $H_2$, or equal to $H_3$ and $H_4$, depending on the verse of the monotonicity of $\mmu$. In any case, the last estimate is equivalent to (\ref{est:Ekt0}).

Let us consider now, more generally, the solution $v_{\lambda,k}:[t_0,+\infty)\to\re$ to equation
\begin{equation}
    v_{\lambda,k}''(t)+\lambda^2 c(t)v_{\lambda,k}(t)=0,
    \nonumber
\end{equation}
with ``initial'' data
\begin{equation}
    v_{\lambda,k}(a_k)=0,
    \qquad
    v_{\lambda,k}'(a_k)=c_\infty^{1/4},
%    \label{eqn:vkl-data}
    \nonumber
\end{equation}
and its energy
\begin{equation}
    E_{\lambda,k}(t):=\frac{v_{\lambda,k}'(t)^2}{c_\infty^{1/2}}+\lambda^2 c_\infty^{1/2} v_{\lambda,k}(t)^2.
    \nonumber
\end{equation}

Since solutions depend continuously on $\lambda$, from (\ref{est:Ekbk}) and (\ref{est:Ekt0}) we deduce that there exists $\widehat{\lambda}_k\in(\lambda_k,\lambda_{k+1})$ such that
\begin{equation}
    E_{\lambda,k}(b_k)\geq
    \frac{1}{2}E_k(b_k)
    \qquad\text{and}\qquad
    E_{\lambda,k}(t_0)\leq 2E_k(t_0)
    \qquad
    \forall\lambda\in[\lambda_k,\widehat{\lambda}_k].
    \label{est:Eklt0}
\end{equation}

Now for every integer $k\geq 1$ we consider the set
\begin{equation}
    \mathcal{M}_k:=\left\{\xi\in\mathcal{M}:\lambda(\xi)\in[\lambda_k,\widehat{\lambda}_k]\right\}.
    \nonumber
\end{equation}

Since $\lambda_k\geq\lambda_0$, by our assumption on the spectrum of $A$ we know that $\mu(\mathcal{M}_k)>0$, and therefore for every $t\geq t_0$ and every $\xi\in\mathcal{M}$ we can set
\begin{equation}
    \widehat{u}(t,\xi):=
    \begin{cases}
        \dfrac{1}{k E_k(t_0)^{1/2}\mu(\mathcal{M}_k)^{1/2}}\cdot
        v_{\lambda(\xi),k}(t)\quad & \text{if $\xi\in\mathcal{M}_k$ for some $k\geq 1$},
        \\
        0 & \text{otherwise}.
    \end{cases}
    \nonumber
\end{equation}

We claim that the function $u(t):=\mathscr{F}^{-1}(\widehat{u}(t,\xi))$ is a solution to (\ref{eqn:main}) that satisfies (\ref{th:counterexample}) with
\begin{equation}
    H_5:=\frac{H_{14}}{2}.
    \nonumber
\end{equation}

To begin with, we have to show that $(u(t_0),u'(t_0))\in D(A^{1/2})\times H$, from which it follows that
\begin{equation}
    u\in C^0\left([t_0,+\infty),D(A^{1/2})\right)\cap C^1([t_0,+\infty),H)
    \nonumber
\end{equation}
due to the regularity of the coefficient $c(t)$. To this end, we observe that
\begin{eqnarray*}
    \E_u(t_0) & = & 
    \dfrac{\|u'(t_0)\|_H^2}{c_\infty^{1/2}}+
    c_\infty^{1/2}\,\|A^{1/2}u(t_0)\|_H^2
    \\[1ex]
    & = &
    \int_{\mathcal{M}}\left(
    \dfrac{|\widehat{u}'(t_0,\xi)|^2}{c_\infty^{1/2}}+
    \lambda(\xi)^2 c_\infty^{1/2}\,|\widehat{u}(t_0,\xi)|^2\right)\,d\xi
    \\[1ex]
    & = &
    \sum_{k=1}^\infty
    \frac{1}{k^2 E_k(t_0)\mu(\mathcal{M}_k)}
    \int_{\mathcal{M}_k}\left(
    \dfrac{|v_{\lambda(\xi),k}'(t_0)|^2}{c_\infty^{1/2}}+
    \lambda(\xi)^2 c_\infty^{1/2}\,|v_{\lambda(\xi),k}(t_0)|^2\right)d\xi.
\end{eqnarray*}

Now from the second relation in (\ref{est:Eklt0}) we know that
\begin{equation}
    \dfrac{|v_{\lambda(\xi),k}'(t_0)|^2}{c_\infty^{1/2}}+
    c_\infty^{1/2}\lambda(\xi)^2\,|v_{\lambda(\xi),k}(t_0)|^2=
    E_{\lambda(\xi),k}(t_0)\leq
    2E_k(t_0)
    \qquad
    \forall\xi\in\mathcal{M}_k,
    \nonumber
\end{equation}
and hence
\begin{equation}
    \E_u(t_0)=
    \dfrac{\|u'(t_0)\|_H^2}{c_\infty^{1/2}}+
    c_\infty^{1/2}\,\|A^{1/2}u(t_0)\|_H^2\leq
    2\sum_{k=1}^\infty\frac{1}{k^2}=
   \frac{\pi^2}{3}.
    \nonumber
\end{equation}

It remains to estimate $\E_u(b_k)$ from below. To this end, by considering only the contribution of $\mathcal{M}_k$, we obtain that
\begin{eqnarray*}
    \E_u(b_k) & = & 
    \dfrac{\|u'(b_k)\|_H^2}{c_\infty^{1/2}}+
    c_\infty^{1/2}\,\|A^{1/2}u(b_k)\|_H^2
    \\
    & = &
    \int_{\mathcal{M}}\left(
    \frac{|\widehat{u}'(b_k,\xi)|^2}{c_\infty^{1/2}}+
    \lambda(\xi)^2c_\infty^{1/2}\,|\widehat{u}(b_k,\xi)|^2\right)\,d\xi
    \\
    & \geq &
    \frac{1}{k^2 E_k(t_0)\mu(\mathcal{M}_k)}
    \int_{\mathcal{M}_k}\left(
    \frac{|v_{\lambda(\xi),k}'(b_k)|^2}{c_\infty^{1/2}}+
    \lambda(\xi)^2c_\infty^{1/2}\,|v_{\lambda(\xi),k}(b_k)|^2\right)d\xi.
\end{eqnarray*}

On the other hand, from the first relation in (\ref{est:Eklt0}) we know that
\begin{equation}
    \frac{|v_{\lambda(\xi),k}'(b_k)|^2}{c_\infty^{1/2}}+
    c_\infty^{1/2}\lambda(\xi)^2\,|v_{\lambda(\xi),k}(b_k)|^2=
    E_{\lambda(\xi),k}(b_k)\geq
    \frac{1}{2}E_k(b_k),
    \nonumber
\end{equation}
from which we obtain that
\begin{equation}
    \E_u(b_k)\geq
    \frac{1}{2k^2 E_k(t_0)}E_k(b_k).
    \nonumber
\end{equation}

Taking into account (\ref{est:Ekbk}) and (\ref{est:Ekt0}), we deduce that
\begin{eqnarray*}
    \frac{\E_u(b_k)}{\E_u(t_0)} & \geq &
    \frac{3}{\pi^2}\cdot
    \frac{1}{2k^2 E_k(t_0)}\cdot
    \frac{1}{2}
    \exp\left(H_{14}M(b_k)^{1/2}\right)
    \\
    & \geq &
    \frac{3}{4\pi^2}\cdot\frac{1}{H_{15}}\cdot\frac{1}{k^2}
    \exp\left(H_{14}M(b_k)^{1/2}-H_{16}M(A_{k})^{1/2}\right).
\end{eqnarray*}

At this point (\ref{th:counterexample}) is proved if we show that
\begin{equation}
    H_{14}M(b_k)^{1/2}-H_{16}M(A_{k})^{1/2}-2\log k-\log\left(\frac{3}{4\pi^2}\cdot\frac{1}{H_{15}}\right)\geq
    \frac{H_{14}}{2}M(b_k)^{1/2}
    \nonumber
\end{equation}
when $k$ if large enough (because we can always remove a finite number of terms from the sequence $b_k$). The latter is true because $M(b_k)\geq 4k^2 M(A_k)$, and $M(A_k)\to +\infty$, so that
\begin{equation}
    M(b_k)^{1/2}\geq
    2kM(A_{k})^{1/2}\geq
    kM(A_{k})^{1/2}+k
    \nonumber
\end{equation}
as soon as $M(A_k)\geq 1$.
\qed

%%\clearpage

\subsubsection*{\centering Acknowledgments}

Both authors are members of the Italian {\selectlanguage{italian}% 
``Gruppo Nazionale per l'Analisi Matematica, la Probabilità e le loro Applicazioni'' (GNAMPA) of the ``Istituto Nazionale di Alta Matematica'' (INdAM)}. 

The first author was partially supported by PRIN 2020XB3EFL, ``Hamiltonian and Dispersive PDEs''. 

The authors acknowledge the MIUR Excellence Department Project awarded to the Department of Mathematics, University of Pisa, CUP I57G22000700001.

\selectlanguage{english}

%%%%%\clearpage

%{\small 
% \bibliographystyle{../../../BibTeX/MaxNew}
% \bibliography{../../../BibTeX/Damping}
%\bibliographystyle{MaxNew}
%\bibliography{Damping}

\begin{thebibliography}{10}
\providecommand{\url}[1]{\texttt{#1}}
\providecommand{\urlprefix}{URL }
\providecommand{\selectlanguage}[1]{\relax}
\providecommand{\eprint}[2][]{\url{#2}}

\bibitem{2012-Ferrara-BoiMan}
\textsc{C.~Boiti}, \textsc{R.~Manfrin}.
\newblock On the asymptotic boundedness of the energy of solutions of the wave equation {$u_{tt}-a(t)\Delta u=0$}.
\newblock \emph{Ann. Univ. Ferrara Sez. VII Sci. Mat.} \textbf{58} (2012), no.~2, 251--289.

\bibitem{2006-JDE-Colombini}
\textsc{F.~Colombini}.
\newblock Energy estimates at infinity for hyperbolic equations with oscillating coefficients.
\newblock \emph{J. Differential Equations} \textbf{231} (2006), no.~2, 598--610.

\bibitem{DGCS}
\textsc{F.~Colombini}, \textsc{E.~De~Giorgi}, \textsc{S.~Spagnolo}.
\newblock Sur les \'{e}quations hyperboliques avec des coefficients qui ne d\'{e}pendent que du temps.
\newblock \emph{Ann. Scuola Norm. Sup. Pisa Cl. Sci. (4)} \textbf{6} (1979), no.~3, 511--559.

\bibitem{2003-BSM-ColDSaRei}
\textsc{F.~Colombini}, \textsc{D.~Del~Santo}, \textsc{M.~Reissig}.
\newblock On the optimal regularity of coefficients in hyperbolic {C}auchy problems.
\newblock \emph{Bull. Sci. Math.} \textbf{127} (2003), no.~4, 328--347.

\bibitem{2007-DIE-DSaKinRei}
\textsc{D.~Del~Santo}, \textsc{T.~Kinoshita}, \textsc{M.~Reissig}.
\newblock Energy estimates for strictly hyperbolic equations with low regularity in coefficients.
\newblock \emph{Differential Integral Equations} \textbf{20} (2007), no.~8, 879--900.

\bibitem{2015-JMAA-EbeFitHir}
\textsc{M.~R. Ebert}, \textsc{L.~Fitriana}, \textsc{F.~Hirosawa}.
\newblock On the energy estimates of the wave equation with time dependent propagation speed asymptotically monotone functions.
\newblock \emph{J. Math. Anal. Appl.} \textbf{432} (2015), no.~2, 654--677.

\bibitem{gg:DGCS-critical}
\textsc{M.~Ghisi}, \textsc{M.~Gobbino}.
\newblock Critical counterexamples for linear wave equations with time-dependent propagation speed.
\newblock \emph{J. Differential Equations} \textbf{269} (2020), no.~12, 11435--11460.

\bibitem{gg:CDSR-optimal}
\textsc{M.~Ghisi}, \textsc{M.~Gobbino}.
\newblock Finite vs infinite derivative loss for abstract wave equations with singular time-dependent propagation speed.
\newblock \emph{Bull. Sci. Math.} \textbf{166} (2021), 102918, 41.

\bibitem{gg:OptDerLoss}
\textsc{M.~Ghisi}, \textsc{M.~Gobbino}.
\newblock Optimal derivative loss for abstract wave equations.
\newblock \emph{Math. Ann.} \textbf{386} (2023), no. 1-2, 455--494.

\bibitem{gg:residual}
\textsc{M.~Ghisi}, \textsc{M.~Gobbino}.
\newblock Three examples of residual pathologies.
\newblock \emph{Jpn. J. Math.} \textbf{18} (2023), no.~1, 67--113.

\bibitem{2007-MathAnn-Hirosawa}
\textsc{F.~Hirosawa}.
\newblock On the asymptotic behavior of the energy for the wave equations with time depending coefficients.
\newblock \emph{Math. Ann.} \textbf{339} (2007), no.~4, 819--838.

\bibitem{2021-JMAA-Hirosawa}
\textsc{F.~Hirosawa}.
\newblock On the energy estimates of semi-discrete wave equations with time dependent propagation speed.
\newblock \emph{J. Math. Anal. Appl.} \textbf{496} (2021), no.~1, Paper No. 124798, 28.

\bibitem{2009-JMAA-HirWir}
\textsc{F.~Hirosawa}, \textsc{J.~Wirth}.
\newblock Generalised energy conservation law for wave equations with variable propagation speed.
\newblock \emph{J. Math. Anal. Appl.} \textbf{358} (2009), no.~1, 56--74.

\bibitem{2005-Hokkaido-ReiSmi}
\textsc{M.~Reissig}, \textsc{J.~Smith}.
\newblock {$L^p$}-{$L^q$} estimate for wave equation with bounded time dependent coefficient.
\newblock \emph{Hokkaido Math. J.} \textbf{34} (2005), no.~3, 541--586.

\bibitem{2007-EJDE-Tarama}
\textsc{S.~Tarama}.
\newblock Energy estimate for wave equations with coefficients in some {B}esov type class.
\newblock \emph{Electron. J. Differential Equations}  (2007), No. 85, 12.

\bibitem{1990-CPDE-Yamazaki}
\textsc{T.~Yamazaki}.
\newblock On the {$L^2({\bf R}^n)$} well-posedness of some singular or degenerate partial differential equations of hyperbolic type.
\newblock \emph{Comm. Partial Differential Equations} \textbf{15} (1990), no.~7, 1029--1078.

\end{thebibliography}
%}

\label{NumeroPagine}

\end{document}